\documentclass[10pt,reqno]{amsart}
\usepackage{amsmath}
\usepackage{amssymb}
\usepackage{amsthm}
\usepackage{stackengine}
\usepackage{scalerel}
\usepackage{graphicx}
\usepackage[table]{xcolor}
\usepackage[inline]{enumitem}
\usepackage{delimset}
\usepackage{hyperref}
\usepackage{cleveref}

\setlist[enumerate]{label={\upshape(\roman*)}}

\newtheorem{theorem}{Theorem}[section]
\newtheorem{lemma}[theorem]{Lemma}
\newtheorem{corollary}[theorem]{Corollary}
\newtheorem{proposition}[theorem]{Proposition}

\theoremstyle{remark}
\newtheorem{definition}[theorem]{Definition}
\newtheorem{example}[theorem]{Example}
\newtheorem{question}{Question}
\newtheorem{problem}[theorem]{Problem}
\newtheorem{remark}[theorem]{Remark}
\newtheorem*{remar}{Remark}

\newcommand\openbigstar[1][0.65]{%
  \scalerel*{%
    \stackinset{c}{0pt}{c}{}{\scalebox{#1}{\color{white}{$\bigstar$}}}{%
      $\bigstar$}%
  }{\bigstar}
}

\newcommand{\BH}{\mathcal{B}(\Hil)}
\newcommand{\KH}{\mathcal{K}(\Hil)}
\newcommand{\Hil}{\mathcal{H}}
\DeclareMathOperator{\diag}{diag}
\DeclareMathOperator{\trace}{tr}
\DeclareMathOperator{\rank}{rank}

\newcommand{\seq}{\delim<>}
\newcommand{\ideal}{\delim()}
\newcommand{\comm}{\delimpair{[}{[.],}{]}}
\newcommand{\I}{\mathcal{I}}
\newcommand{\J}{\mathcal{J}}
\newcommand{\schatten}[1]{\mathcal{C}_{#1}}
\newcommand{\traceclass}{\schatten{1}}
\newcommand{\charset}{\Sigma}
\renewcommand{\star}{\openbigstar}
\newcommand{\czstar}{c_0^{\star}}

\newcommand{\term}[1]{\emph{#1}}

\numberwithin{equation}{section}

\begin{document}
\title[On Commutators of Compact Operators]{On commutators of compact operators via block tridiagonalization: generalizations and limitations of Anderson's approach}

\author{Jireh Loreaux}
\address{J.~Loreaux, Department of Mathematics and Statistics, Southern Illinois University Edwardsville, Edwardsville, IL 62026 (USA)}
\email{jloreau@siue.edu}
\author{Sasmita Patnaik*}
\address{S.~Patnaik, Department of Mathematics and Statistics, Indian Institute of Technology, Kanpur 208016 (INDIA)}
\email{sasmita@iitk.ac.in}
\author{Srdjan Petrovic}
\address{S.~Petrovic, Department of Mathematics, Western Michigan University, Kalamazoo, MI 49008 (USA)}
\email{srdjan.petrovic@wmich.edu}
\author{Gary Weiss**}
\address{G.~Weiss, Department of Mathematics, University of Cincinnati, Cincinnati, OH 45221-0025 (USA)}
\email{gary.weiss@uc.edu}

\thanks{*Supported by Science and Engineering Research Board, Core Research Grant 002514 \quad\\
  **Partially supported by Simons Foundation collaboration grants 245014 and 636554}

\maketitle

\begin{abstract}
  We offer a new perspective and some advances on the 1971 Pearcy--Topping problem: Is every compact operator a commutator of compact operators?
  Our goal is to analyze and generalize the 1970's work in this area of Joel Anderson combined with the work of the last named author of this paper.
  We reduce the general problem to a simpler sequence of finite matrix equations with norm constraints, while at the same time developing strategies for counterexamples.
  Our approach is to ask which compact operators $T$ are commutators $AB-BA$ of compact operators $A,B$; and to analyze the implications of Joel Anderson's contributions to this problem, which will yield a generalization of his method.
  By extending the techniques of Anderson \cite{And-1977-JRAM}  we obtain new classes of operators that are commutators of compact operators beyond those obtained in \cite{Pat-2012} and \cite{BPW-2014-VLOT}.
  And by employing the techniques of the last named author \cite{Wei-1980-IEOT}, we found obstructions to extending Anderson's techniques in terms of certain constraints for $T$, with special focus on when $T$ is a strictly positive compact diagonal operator.
  Some of these constraints involve general universal block tridiagonal matrix forms for operators, and some involve $\BH$-ideal constraints.
  And in terms of these matrix forms, we give some equivalences, some sufficient conditions and some necessary conditions for this Pearcy--Topping problem and its various offshoots to hold true.
  These matrix forms are a sparsification of matrix representations of an operator (an increase in the proportion of zeros in its corners by a change of basis) and we measure the support density of these forms.
  And finally we provide some necessary conditions for the Pearcy--Topping problem involving singular numbers and $\BH$-ideal constraints.
\end{abstract}

\section{Introduction and Historical Survey}
The  \textbf{classical commutator problem} asks which $\BH$ operators are commutators, i.e., operators of the form
$\comm{A}{B}=AB-BA$, for $\Hil$ an infinite dimensional, separable, complex Hilbert space.
This was settled by Brown and Pearcy \cite{BP-1965-AoM2}: An operator is a commutator if and only if it it is not a thin operator, i.e.,
not of the form $\lambda+K$, where $\lambda$ is a nonzero complex number and $K$ is compact.
Following the Brown--Pearcy characterization of $\BH$-commutators \cite{BP-1965-AoM2}, much of the research on single commutators was focused on special situations when operators $A,B$ belong to various ideals.
Later research mostly focused on commutator ideals in $\BH$, $\comm{\I}{\J}$, i.e., the finite  linear span of commutators of operators from $\I$ with operators from $\J$, in part because of the connection of $\comm{\I}{\BH}$ to traces.
(Chronologically, see \cite{PT-1971-MMJ,Wei-1975,And-1977-JRAM,Wei-1980-IEOT,Kal-1989-JFA,DFWW-2004-AM,KW-2010-JOT,DK-2018-PAMS}.)
The general characterization of commutator ideals $\comm{\I}{\J}$,  was obtained in \cite[Theorem 5.6]{DFWW-2004-AM} (or more simply distilled from \cite[Introduction pp. 2--3 and remark following]{DFWW-2004-AM}): a normal operator $T\in \I\J$
belongs to $\comm{\I}{\J}$ if and only if $T$ belongs to $\comm{\I\J}{\BH}$ if and only if
\begin{equation*}
  \abs*{\dfrac{\lambda_1+\lambda_2+\dots+\lambda_n}{n}} = O(\mu_n).
\end{equation*}
Here  $\I,\J$ are arbitrary $\BH$-ideals at least one of which is proper,
$\nu_n, \mu_n$ are monotone sequences of nonnegative numbers converging to 0 where $\diag(\nu_n), \diag(\mu_n)$ belong to the product ideal  $\I\J$,
and
$T$ has eigenvalue sequence $\seq{\lambda_n}$, the nonzero ones counted according to their multiplicities but taken in any order subject only to the condition that $\abs{\lambda_n}=O(\nu_n)$.
(The existence of such an order and sequence $\nu$ is equivalent to the condition $T \in \I\J$, which follows easily from the methods in \cite{DFWW-2004-AM}.)
A useful version of this characterization  can be stated for
any positive operator $T\in \I\J$ using its $s$-numbers (singular value sequence): a positive $T$ belongs to $\comm{\I}{\J}$ if and only if the diagonal operator
\begin{equation*}
  \diag \seq*{\dfrac{s_1+\dots+s_n}{n}} \in \I\J.
\end{equation*}
The fact that $s$-numbers are positive while eigenvalues may not be was important in \cite{DFWW-2004-AM} and \cite{KW-2010-JOT} and related work of the authors of \cite{KW-2010-JOT}, primarily because working on positive cones, meaning cones of positive operators was essential for some of their
results.

The single commutator problem for compact operators is another commutator question posed by Pearcy and Topping in 1971 which at the time of this writing remains open.

\begin{question}
  Is every compact operator a commutator of compact operators?
\end{question}

Seminal results on this and related questions appeared in the 1970's by Anderson \cite{And-1977-JRAM}, results by the last named author of this paper  \cite{Wei-1980-IEOT} settling \cite[Problem 3(i)]{PT-1971-MMJ}, and a natural test question posed by him.
Motivated by Anderson's rank one projection result in  \cite{And-1977-JRAM}, he asked: Can a strictly positive compact operator be a commutator of compact operators?
There followed 30 years later an affirmative solution to this test case in the unpublished work of Davidson--Marcoux--Radjavi \cite{DMR} and independently a little later, similar results of Beltita joint with the second and fourth authors herein \cite{BPW-2014-VLOT,Pat-2012}.
A little earlier, in  1994, L.~Brown, using block Hessenberg forms, had proved in \cite{Bro-1994-JRAM} that  no finite rank nonzero trace operator is a commutator of operators in Schatten classes $\schatten{p}$ and  $\schatten{q}$, when $1/p+1/q\ge 1/2$.
(In particular, there are operators in $\schatten{2}$ that are not  commutators of $\schatten{4}$ operators).
Then in  \cite[Section 7]{DFWW-2004-AM}  the authors in 2004 obtained a sufficient condition in terms of ideals $\I,\J$ in $\BH$ for an operator to be a single commutator of operators from $\I$ and $\J$ (we denote this class by $\comm{\I}{\J}_1$) in the commutator ideal $\comm{\I}{\J}$.
And recently in \cite{DK-2018-PAMS}, K.~Dykema and A.~Krishnaswamy-Usha proved that all nilpotent compact operators are commutators of compact operators.
From this, focus began on trying to represent specific compact operators $T$ and classes of compact operators as special kinds of commutators.
Constructions for these, it was hoped, would lead to a more general strategy which began with \cite{DMR,Pat-2012,BPW-2014-VLOT} chronologically.

This paper focuses mainly on representing normal compact operators (i.e., diagonal compact operators), and particularly strictly positive ones, as commutators of compact operators.
The latter are notoriously hard to represent, and the only known  diagonal sequence patterns obtainable have either infinitely many zeros (Anderson) or have each strictly posiitive eigenvalue repeated according to a certain rigid multiplicity structure (Davidson--Marcoux--Radjavi and Beltita--Patnaik--Weiss, see herein \Cref{t6,thm3.4'}). 
One of the central accomplishments of this paper (\Cref{sec3}) is
a modification of these constructions leading to the affirmative answer to a question arising from \cite[p. 9]{BPW-2014-VLOT}:
there are infinitely many strictly positive compact operators with distinct nonzero eigenvalues that are commutators of compact operators, thereby expanding the known class.
We hope from these constructions more general ones might evolve, perhaps leading to a full characterization of which normal compact operators are commutators of compact operators.

Next we address the question whether Anderson's commutator model \cite{And-1977-JRAM} can be used to obtain \emph{every} positive compact operator.
In \Cref{sec3a} we present several obstructions limiting this model.
Some of them have the flavor of majorization type constraints involving arithmetic means arising in the work of Ky Fan  \cite{Fan-1951-PNASUSA}.

To keep manageable the computations, following Anderson's approach, it was natural to use sparse matrices (sparse meaning with \emph{most or at least many} entries zero).
For instance, the commutator of an upper and a lower weighted shift is a diagonal operator, and this phenomenon has led to many important consequences in the early literature.
To sparsify a matrix (change basis to increase the zero regions) for a general $T$, and in particular for $T=AB-BA$, we construct an orthonormal basis for which the matrix of $T$ is {sparse}
and for which, when needed,
simultaneously, the $A$ and $B$ matrix forms must also be  sparse and of the same especially convenient form.
And especially when $A,B$ are sparse of this convenient form,
the computations of $AB$ and $BA$ become more manageable.
A change of basis corresponds to a unitary operator $U$, which yields the matrix  representation of $T$ in the new basis $\seq{Ue_n}$ as $U^{-1}TU$ in the original basis $\seq{e_n}$ of $T$.
This leads to the question: how sparse can $U^{-1}TU$ be made?
In \Cref{sec2} we lift from \cite{PPW-2020-TmloVLOt} a brief description of certain matrix sparsifications.
An in-depth study on this can be found there.
\Cref{sec2} also provides a new Pearcy--Topping commutator problem equivalence (\Cref{equivthm}) based on this that in a sense fully generalizes Anderson's approach.
Historically, Anderson's tour de force solved affirmatively the Pearcy--Topping test question that every rank one projection operator is a commutator of compact operators, and used this to prove that every compact operator is  a commutator of a compact operator with a bounded operator.
For this rank one projection case he employed block tridiagonal matrices with block sizes growing arithmetically.
And subsequently motivated by Anderson's model, our \cite{Wei-1980-IEOT} introduced general staircase forms which led us to general block tridiagonal forms with further matrix sparsifications in our precursor \cite{PPW-2020-TmloVLOt} to this paper.

A word about evolving block tridiagonal terminology used herein.
The displays \eqref{eq6}--\eqref{eq5} in \Cref{sec3} constitute Anderson's model (AM).
In \Cref{sec3a} we generalize (AM) to an exponential Anderson's model (EAM) by choosing instead exponentially growing block sizes whereas (AM) uses arithmetically growing block sizes.
And after that we use our universal block tridiagonal forms which generalize both.
The exponentially growing block sizes in our universal block tridiagonal forms apply to all operators and can be derived from any arbitrary basis.

In short, the organization of this paper is as follows.
In \Cref{sec3} we reiterate the classical Anderson's model that he used to obtain a rank one projection as a commutator of compact operators.
Our main result (\Cref{t7}), expanding on Anderson's model, is that with a controlled perturbation of his nonzero entries one obtains uncountably many diagonal compact operators  with \emph{distinct} strictly positive entries.
Along with their unitary equivalents, this expands the known class and solves a problem arising from \cite{BPW-2014-VLOT} that evolved from \cite{And-1977-JRAM,DMR,Pat-2012}.

In \Cref{sec3a} we show that not every compact diagonal operator can be obtained using Anderson's construction.
An obstruction comes from a necessary trace condition for Anderson's model to work.
We formulate this condition in two different ways, one of them having a flavor of majorization such as appears in Ky Fan's theorem \cite{Fan-1951-PNASUSA} (or see \cite[Lemma 4.1]{GK-1969-ITTTOLNO}).
We found that
the arithmetic growth of block sizes of Anderson's block tridiagonal form is a main culprit.
Using our universal block tridiagonal forms which have exponential block growth sizes, we consider an expanded version of Anderson's model using instead exponential growth size blocks to provide more degrees of freedom.
But even when the blocks are chosen to grow exponentially, we prove that this particular construction offers no improvement.

The study of this exponential Anderson model ties in naturally with the content of \Cref{sec2} where we, in a sense, fully generalize Anderson's model.
We show that given an abitrary finite set of operators and an arbitrary basis, one can derive from these a new basis so that all the operators simultaneously take block tridiagonal form with identical block sizes.
And these central (and hence all) block sizes grow exponentially (\Cref{t1,T1,t11}).
As a consequence we obtain a block tridiagonal equivalence to the Pearcy--Topping compact commutator problem (\Cref{equivthm}).
This result highlights the value of generalizing Anderson's construction.
But at the same time also explains in a sense its limitations.
And explains the fact, established in \Cref{sec3a}, that the obstruction detected there remains in force when the finite diagonal matrix blocks (with arithmetic growth) of the block tridiagonal operators employed are instead full finite matrix blocks.
(Observe that Anderson's block tridiagonal matrix solutions possess merely two nonzero block diagonals.
Whereas here we consider fully general block tridiagonal forms.)
We present several modifications of the construction in \Cref{t1} which, in addition to obtaining the block tridiagonal forms (\Cref{t1,T1,t11}), guarantee that additionally some of the entries in the off-diagonal blocks can further be made 0, i.e., one can achieve further sparsificaton (\Cref{T2,t3aa}).

In \Cref{secdens} we make a rough comparison between the sparsity of the matrix forms discussed herein (their possible nonzero entries called the support entries) and that of other well-known canonical matrix forms for operators.
We define the basis dependent concept of support density (a measure of sparsification), and we calculate it for various forms.
In particular, we show that for our block tridiagonal forms used here and for the staircase forms developed by the last named author in \cite{Wei-1980-IEOT}, the support density is $2/3$, but that one of our modifications here of our original block tridiagonal forms reduce the support density to $1/2$ (see \Cref{sec5a,thm:arbitrary-staircase-density}).
In contrast, in the bases diagonalizing compact normal operators or having Anderson model forms, the support density is $0$.
And for triangular or Hessenberg matrices, the support density is $1/2$.
But of course not all operators have bases that admit these forms (which for Hessenberg matrices and selfadjoint operators, can depend on the existence of a cyclic vector (see paragraphs preceding \Cref{sec5a})).
But all operators do have universal block tridiagonal forms in some basis.
In fact, from any basis one can derive a new basis to achieve this.

\Cref{sec6} provides many necessary conditions for the Pearcy--Topping problem involving singular numbers (alias $s$-numbers) and algebraic $\BH$-ideal constraints.
 

\section{Anderson's Model (AM)}
\label{sec3}

This section presents our first steps in extending the contributions of \cite{And-1977-JRAM} and \cite{Wei-1980-IEOT} beyond the previous preliminary work (listed chronologically) of \cite{DMR,Pat-2012,BPW-2014-VLOT}.
We start with some beginning history, then proceed to solutions of some test problems that later arose from that history, namely, about which compact positive diagonal operators are commutators of compact operators.

Our main motivation for studying sparse matrices starting in \cite{PPW-2020-TmloVLOt} with a description in \Cref{sec2} herein, comes from the contributions of \cite{And-1977-JRAM} and \cite{Wei-1980-IEOT} to the 1971 questions of Pearcy and Topping \cite{PT-1971-MMJ}.
In particular:
\begin{enumerate*}
\item Is every compact operator a single commutator of compact operators; and 
\item Is every trace class, trace zero operator a commutator of Hilbert--Schmidt operators, or at least a finite sum of such commutators?
\end{enumerate*}
These turned out to be the most important among several results and seminal open questions about commutators and their linear spans listed in \cite{PT-1971-MMJ}.
For starters, they proposed a surprisingly simple test question which turned out to be groundbreaking: Is every rank one projection a single commutator of compact operators?
An affirmative answer was provided by Anderson in \cite{And-1977-JRAM} along with applications (\Cref{t3,t4,t5} herein).
Via weighted shifts, it is easy to obtain selfadjoint operators with initial sums of its eigenvalues converging to $0$ as commutators of compact operators (see \cite[Theorem~1]{FF-1980-PAMS}).
It turned out, for nonzero positive compact operators (hence initial sums of its eigenvalues do not converge to $0$), representing them as a commutator of compact operators was very hard.
This is one of the reasons why targeting positive compact operators became so important.

\begin{theorem}[Anderson \cite{And-1977-JRAM}, 1977]
  \label{t3}
  For $P$ a rank one selfadjoint projection, there exist compact operators $C$ and $Z$ for which $P=CZ-ZC$.
\end{theorem}

An application of this to a more general commutator result is:
\begin{theorem}[Anderson \cite{And-1977-JRAM}, 1977]
  \label{t4}
  If $T$ is a compact operator with $\ker T$ containing an infinite dimensional reducing subspace of $T$, then there are compact operators $C$ and $Z$ for which $T= CZ-ZC$.
\end{theorem}
Regarding the Pearcy--Topping commutator problem for general compact operators, a weaker but completely general result is:
\begin{theorem}[Anderson \cite{And-1977-JRAM}, 1977]
  \label{t5}
  If $T$ is a compact operator, then there is a compact operator $C$ and a bounded operator $Z$ for which  $T= CZ-ZC$.
\end{theorem}

\subsection*{Anderson's Model (AM)}

Historically, many results about commutators employed matrices that were weighted shifts or their adjoints, possibly with block weights (replacing scalar weights).
Anderson used the following block matrix forms, with central block sizes growing arithmetically:

\begin{equation}
  \label{eq6}
  C =
  \begin{pmatrix}
    0 &  A_1  &0  &\dots          \\
    B_1 & 0 & A_2   &\ddots           \\
    0 &B_2 & 0 &\ddots            \\
    \vdots &\ddots &\ddots  &\ddots
  \end{pmatrix},\quad\quad
  Z =
  \begin{pmatrix}
    0 &  X_1  &0  &\dots          \\
    Y_1 & 0 & X_2   &\ddots           \\
    0 & Y_2 & 0 &\ddots            \\
    \vdots &\ddots &\ddots  &\ddots
  \end{pmatrix}.
\end{equation}
The sizes of these diagonal square blocks grow arithmetically: $1\times 1, 2\times 2, 3\times 3, \dots$.
And the  nonzero rectangular blocks grow arithmetically and are all shift-like, i.e., with nonzero entries only on the certain diagonals:
\begin{equation}
  \label{eq5}
  \begin{gathered}
    A_n=
    \begin{pmatrix}
      a_{1,n} & 0 & \cdots   & 0 & 0 \\
      0 & a_{2,n}  & \ddots  & \vdots & \vdots \\
      \vdots &\ddots & \ddots & 0 & \vdots \\
      0 & \cdots & 0 & a_{n,n} & 0
    \end{pmatrix},\quad
    X_n=
    \begin{pmatrix}
      0 & x_{1,n} & 0  & \cdots & 0 \\
      \vdots & 0 & x_{2,n} & \ddots & \vdots \\
      \vdots & \vdots & \ddots & \ddots & 0 \\
      0 & 0 & \cdots & 0 & x_{n,n}
    \end{pmatrix},\\
    \\
    B_n=
    \begin{pmatrix}
      0 & \cdots &\cdots & 0 \\
      -b_{1,n} & 0 & \cdots & 0 \\
      0 &-b_{2,n}  & \ddots & \vdots \\
      \vdots & \ddots & \ddots & 0 \\
      0 & \cdots & 0 & -b_{n,n}
    \end{pmatrix},\quad
    Y_n=
    \begin{pmatrix}
      y_{1,n} & 0  & \cdots  & 0 \\
      0 & y_{2,n} & \ddots & \vdots \\
      \vdots & \ddots & \ddots & 0 \\
      0 & \cdots & 0 & y_{n,n}\\
      0 & \cdots & \cdots & 0
    \end{pmatrix}.
  \end{gathered}
\end{equation}
Throughout the paper we will refer to operators $C$ and $Z$ described by \eqref{eq6}--\eqref{eq5} as the (classical) Anderson model (AM).

The specific matrix structure of the operators $C$ and $Z$ in \eqref{eq6} allowed Anderson to reduce the operator equation $CZ-ZC=D$, for $D$ block diagonal (in Anderson's case $D = P$ the rank one projection) to a system of block matrix equations.
These equations can be split into two groups: those that ensure that $CZ-ZC$ is block diagonal, and those that force those diagonal blocks to match the diagonal blocks $\seq{D_n}$ of $D$.
These sets of equations are
\begin{equation}
  \label{eqD_N}
  \begin{gathered}
    A_nX_{n+1}=X_nA_{n+1}\quad\mbox{ and }\quad B_{n+1}Y_n=Y_{n+1}B_n,\\
    D_1=A_{1}Y_{1}-X_{1}B_{1}, \quad
    D_{n+1}=B_nX_n-Y_nA_n+A_{n+1}Y_{n+1}-X_{n+1}B_{n+1},\quad n=1,2,\dots.
  \end{gathered}
\end{equation}
\begin{remark}
  To solve for $D = CZ-ZC$ where $D$ is block diagonal with central block sizes $1\times 1, 2\times 2, 3\times 3, \dots$ and with $C,Z$ having matrix forms \eqref{eq6}, it is necessary and sufficient to simply solve  the system \eqref{eqD_N} for these finite blocks, so that the $C,Z$ finite block norms all tend to 0.
  (The latter condition ensures $C,Z$ are compact, but moreover it is also a necessary condition by \Cref{thm:bandable}.)
\end{remark}

\begin{remark}
  \label{r2.5}
  Displays \eqref{eq6}--\eqref{eq5} being what we mean by Anderson's model, in \Cref{sec3a} we generalize (AM) to an exponential Anderson's model (EAM) by choosing instead exponentially growing block sizes.
  And later we employ our universal block tridiagonal forms in \Cref{sec2} which generalize both by showing every operator in some basis is block tridiagonal with universal block sizes growing exponentially.
\end{remark}

Based on Anderson's findings above, the last named author asked in 1976 whether any strictly positive compact operator could be a commutator of compact operators (see \cite[p.\,2, Acknowledgement and prior paragraph]{BPW-2014-VLOT}).
This, based on staircase matrix form calculations from \cite[Theorem 2 and lead up]{Wei-1980-IEOT}, in the hope of obtaining a class of counterexamples.
And also because if the kernel of $D$ has infinite dimension, then $D$ is a commutator of compact operators by Anderson's \Cref{t4} above.
But the finite dimensional kernel case for nonnegative $D$ remains unknown, in particular the zero dimensional kernel case, that is, the strictly positive case stated above.
Of course, a positive solution to the strictly positive case trivially yields a positive solution to the general positive case.

It is natural to ask whether the construction that Anderson has employed can be modified to answer this commutator question or, more generally, for all compact normal operators $D$ (i.e., diagonal compact operators).
One such advance came in \cite{BPW-2014-VLOT}.

\begin{theorem}[\protect{\cite[Theorem 3.1]{BPW-2014-VLOT}}]
  \label{t6}
  Positive compact operators are commutators of compact operators when they have eigenvalue sequences:
  \begin{equation*}
    \seq*{d_1, \dfrac{d_2-d_1}{2}, \dfrac{d_2-d_1}{2},\dfrac{d_3-d_2}{3},\dfrac{d_3-d_2}{3},\dfrac{d_3-d_2}{3},\dots},
  \end{equation*}
  where $\seq{d_n}$ is a nondecreasing sequence of nonnegative numbers but $d_n/n \to 0$.
\end{theorem}

This was earlier proved by Davidson, Marcoux, and Radjavi \cite{DMR}, but the result was never published nor disseminated.
But at the time in a private communication to the last named author of this paper they asked if knowing this result would lead to achieving all strictly positive diagonal operators being so representable as commutators of compact operators.
We see below in \Cref{sec3a} (\Cref{tEAM} and \Cref{NOTtEAM}) that not all strictly positive compact operators can be representable this way using Anderson's model, nor by using an expanded version, the exponential Anderson's model (EAM).

An equivalent formulation of \Cref{t6} appears in \cite{Pat-2012}.

\begin{theorem}[\protect{\cite[Theorem 3.2.1]{Pat-2012}}]
  \label{thm3.4'}
  Positive compact operators are commutators of compact operators when
  they have eigenvalue sequences:
  \begin{equation*}
    \seq*{d_1, \dfrac{d_2}{2}, \dfrac{d_2}{2},\dfrac{d_3}{3},\dfrac{d_3}{3},\dfrac{d_3}{3},\dots},
  \end{equation*}
  where $\frac1{n}\sum_{j=1}^n d_j \to 0$.
\end{theorem}

The proof of \Cref{t6} is constructive, with a clever alteration of the Anderson's matrices \eqref{eq5} by inserting a factor $\sqrt{d_n}$ in front of $A_n$, $X_n$, $B_n$ and $Y_n$.

\Cref{t4} had a severe restriction that the target operator must have an infinite dimensional reducing subspace on which it is 0, while \Cref{t6} (and \Cref{thm3.4'}) requires that the eigenvalues repeat.
Therefore, it is of interest to ask whether for all positive or all normal compact operators one can move beyond these two patterns of diagonal sequences.
In particular, is it possible to obtain a diagonal operator with the diagonal entries distinct positive numbers?
Our main result in this section (\Cref{t7}) is that a modification of Anderson's model yields an affirmative answer to this question.

The proof of this theorem relies on the following two lemmas.
Throughout this paper we use the term positive number to mean strictly positive.
\begin{lemma}
  Let $0<L<1$.
  Then there exists a sequence $\seq{\alpha_n}$ of positive numbers such that
  \begin{equation}
    \label{eq32}
    \alpha_{n+1}\ge \left(1+\dfrac{L}{n}\right)\alpha_n,\quad\mbox{ and }\quad \dfrac{\alpha_n}{n}\to 0.
  \end{equation}
\end{lemma}
\begin{proof}
  Let $L<M<1$.
  We define the sequence $\seq{\alpha_n}$ inductively.
  Let $\alpha_1>0$.
  If $\alpha_1,\dots,\alpha_n$ have been selected, we choose $\alpha_{n+1}$ so that
  \begin{equation*}
    \left(1+\dfrac{L}{n}\right)\alpha_n \le \alpha_{n+1}\le \left(1+\dfrac{M}{n}\right)\alpha_n.
  \end{equation*}
  Then $\seq{\alpha_n}$ satisfies the first condition in \eqref{eq32} and
  \begin{equation*}
    \alpha_{n+1}\le \prod_{k=1}^n\left(1+\dfrac{M}{k}\right)\alpha_1.
  \end{equation*}
  But it is well known that the product above is asymptotically equivalent to $n^M$, (i.e., $\prod_{k=1}^n\left(1+{M}/{k}\right)/n^M$ is bounded above and below via taking logarithms) so the second condition in \eqref{eq32} holds, and the lemma is proved.
\end{proof}

\begin{lemma}
  \label{lemma2.7new}
  Let $\seq{\varepsilon_n}_{n=1}^\infty$ be any bounded sequence of strictly positive numbers with $L:=\sup \varepsilon_n<1$, let $\varepsilon_0=0$, and let $\seq{\alpha_n}_{n=1}^\infty$ be any sequence of positive numbers satisfying \eqref{eq32}.
  Then there exists a sequence of strictly positive numbers $\seq{d_n}_{n=1}^\infty$ with the following properties:
  \begin{enumerate}
  \item
    \begin{equation}
      \label{eq89}
      0<d_1\le\alpha_1.
    \end{equation}
  \item For all $n\ge 1$ and $1\le k\le n+1$
    \begin{equation}
      \label{eq22}
      d_n \dfrac{1+\varepsilon_{k-1}/n}{1+\varepsilon_{k}/(n+2)}<d_{n+1}\le \alpha_{n+1}.
    \end{equation}
  \end{enumerate}
\end{lemma}

\begin{proof}
  We will construct the sequence $\seq{d_n}$ inductively.
  Let $d_1$ be any real number satsifying condition \eqref{eq89}.
  Suppose that the numbers $d_1,d_2,\dots,d_n$ have been selected so that they satisfy conditions of the form \eqref{eq89}--\eqref{eq22}.
  Notice that, in view of $0<\varepsilon_n\le L$ as well as the induction hypothesis, we have for $n\ge 1$ that
  \begin{equation*}
    d_n \dfrac{1+\varepsilon_{k-1}/n}{1+\varepsilon_{k}/(n+2)}<d_n\left(
      1+\dfrac{L}{n}\right)\le \alpha_n\left(
      1+\dfrac{L}{n}\right)\le \alpha_{n+1}.
  \end{equation*}
  Thus, $d_{n+1}$ can be chosen anywhere in the interval
  \begin{equation}
    \label{eq79}
    \left(d_n \dfrac{1+\varepsilon_{k-1}/n}{1+\varepsilon_{k}/(n+2)}, \alpha_{n+1}\right]. \qedhere
  \end{equation}
\end{proof}

Now we can prove the promised theorem.
\begin{theorem}
  \label{t7}
  Let $\seq{\varepsilon_n}_{n=0}^\infty$, their supremum
  $L$, and $\seq{\alpha_n}_{n=1}^\infty$ be as in \Cref{lemma2.7new}, and let $\seq{d_n}_{n=1}^\infty$ be a sequence provided by this lemma.

  Set $d_{1,1}=d_1 + \varepsilon_1/2$ and for all $n\ge 1$ and $1\le k\le n+1$ and set
  \begin{equation}
    \label{eq12}
    d_{k,n+1}=\frac{d_{n+1}-d_n}{n+1}+\frac{\varepsilon_k d_{n+1}}{(n+1)(n+2)}-\frac{\varepsilon_{k-1}d_n}{n(n+1)}.
  \end{equation}
  Then the diagonal operator $D = \bigoplus D_n$, where $D_n = \diag \seq{d_{k,n}}_{k=1}^n$, is strictly positive and is a commutator of compact operators.
  Moreover, for many such $D$, all the diagonal entries $\seq{d_{k,n}}$ are distinct.
\end{theorem}

\begin{proof}
  We first show that $d_{k,n}>0$ for all $1 \le k \le n$.
  Indeed, $d_{1,1}=d_1 + \varepsilon_1/2 > 0$.
  And for $n\ge 1$ and $1 \le k \le n+1$,
  \begin{equation*}
    d_{k,n+1}=\frac{d_{n+1}-d_n}{n+1}+\frac{\varepsilon_k d_{n+1}}{(n+1)(n+2)}-\frac{\varepsilon_{k-1} d_n}{n(n+1)}>0
  \end{equation*}
  is equivalent to the inequality
  \begin{equation*}
    d_{n+1}\left(1+\frac{\varepsilon_k}{n+2}\right)> d_n\left(1+\frac{\varepsilon_{k-1}}{n}\right)
  \end{equation*}
  which follows from the left-hand inequality in \eqref{eq22}.
  Thus $D$ is strictly positive.

  In order to establish that $D$ is a commutator of compact operators, we will rely on the construction used in the proof of \Cref{thm3.4'} but with a controlled perturbation of matrices in \eqref{eq5}.
  Indeed, Anderson's original parameters were (after setting his parameter $t=1/2$)
  \begin{equation*}
    \begin{aligned}
      a_{k,n}&=\frac{\sqrt{n+1-k}}{n},\\
      b_{k,n}&=\frac{\sqrt{k}}{n+1},
    \end{aligned}
    \quad \begin{aligned}
      x_{k,n}&=\frac{\sqrt{k}}{n}, \\
      y_{k,n}&=\frac{\sqrt{n+1-k}}{n+1}.
    \end{aligned}
  \end{equation*}
  We define the numbers:
  \begin{equation}
    \label{eq2}
    \begin{gathered}
      a_{k,n}=\sqrt{d_n}\frac{\sqrt{n+1-k}}{n},\quad x_{k,n}=\sqrt{d_n}\frac{\sqrt{k+\varepsilon_k}}{n}, \\
      b_{k,n}=\sqrt{d_n}\frac{\sqrt{k+\varepsilon_k}}{n+1}, \quad
      y_{k,n}=\sqrt{d_n}\frac{\sqrt{n+1-k}}{n+1},
    \end{gathered}
  \end{equation}
  for $n\in\mathbb{N}$ and $1\le k\le n$;
  the matrices $A_n,X_n, B_n,Y_n$ as in \eqref{eq5};
  the operators $C,Z$ as in \eqref{eq6}.
  Next we show $C,Z$ are compact and $D=CZ-ZC$ for $D$ as defined above.

  It is easy to verify that $C$ and $Z$ are compact.
  Indeed, the inequalities $\abs{a_{k,n}},\abs{b_{k,n}},\abs{x_{k,n}},\abs{y_{k,n}}\le \sqrt{d_n/n}$ are obvious and they imply that $\norm{A_n},\norm{B_n},\norm{X_n},\norm{Y_n}\le \sqrt{d_n/n}$.
  In view of the right-hand inequality in \eqref{eq22} and the second condition in \eqref{eq32}, i.e., $0\le d_{n+1}\le \alpha_{n+1}$ and $\alpha_n/n\to 0$, we obtain that $\norm{A_n},\norm{B_n},\norm{X_n},\norm{Y_n}\to 0$ as $n\to\infty$, so $C$ and $Z$ are compact.

  A calculation shows that off-diagonal blocks of $CZ-ZC$ are of the form $A_nX_{n+1}-X_nA_{n+1}$ and $B_{n+1}Y_n-Y_{n+1}B_n$.
  It follows that they all vanish if and only if
  \begin{equation}
    \label{eq2new} a_{k,n}x_{k,n+1}=x_{k,n}a_{k+1,n+1}\quad\mbox{ and }\quad
    b_{k,n+1}y_{k,n}=y_{k+1,n+1}b_{k,n},
  \end{equation}
  for $n\in\mathbb{N}$ and $1\le k\le n$.
  A straightforward verification that all of these hold is left to the reader.

  Since the off-diagonal blocks of $CZ-ZC$ are zero, $CZ-ZC$ is a block diagonal matrix.
  And using the fact that $C$ and $Z$ have the forms \eqref{eq6} it is straightforward to verify that the diagonal  blocks of $D$ are given by the formulas \eqref{eqD_N}
  \begin{equation}
    \label{eq77}
    D_{n+1}=A_{n+1}Y_{n+1}-X_{n+1}B_{n+1}+B_nX_n-Y_nA_n,\quad n\ge 0,
  \end{equation}
  with the convention that $A_0=B_0=X_0=Y_0=0$.
  Thus, the first ($1\times 1$) block is
  \begin{equation}
    \label{eq3new}
    D_1=A_1Y_1-X_1B_1=\begin{pmatrix} a_{1,1}y_{1,1}+x_{1,1}b_{1,1} \end{pmatrix}=
    \begin{pmatrix} \frac{d_1}{2}+\frac{d_1(1+\varepsilon_1)}{2} \end{pmatrix}
    =\begin{pmatrix} d_1+\frac{\varepsilon_1}{2} \end{pmatrix}=\begin{pmatrix} d_{1,1} \end{pmatrix}.
  \end{equation}
  And using \eqref{eq77}, it is an exercise in matrix algebra to show that, for $n\ge 1$, the $(n+1)\times (n+1)$ diagonal block of $CZ-ZC$ equals
  \begin{equation*}
    D_{n+1}=
    \frac{d_{n+1}-d_n}{n+1} I_{n+1}+
    \begin{pmatrix}
      \frac{\varepsilon_1 d_{n+1}}{(n+1)(n+2)}&&&\\&\frac{\varepsilon_2 d_{n+1}}{(n+1)(n+2)}-\frac{\varepsilon_1 d_n}{n(n+1)}&&\\&&\ddots&\\&&&\frac{\varepsilon_{n+1} d_{n+1}}{(n+1)(n+2)}-\frac{\varepsilon_n d_n}{n(n+1)}
    \end{pmatrix},
  \end{equation*}
  where the symbol $I_{n+1}$ denotes the $(n+1)\times(n+1)$ identity matrix.
  Thus each $(n+1)\times (n+1)$ diagonal block  is itself a diagonal matrix.
  A comparison with \eqref{eq12} shows that these diagonal entries are exactly $\seq{d_{k,n+1}}_{k=1}^{n+1}$ which proves that $D$ is a commutator of compact operators.

  Finally, we note that it is often the case that $d_{k,n}$ are all distinct.
  Indeed, notice that in the inductive construction in the proof of \Cref{lemma2.7new} it suffices to choose $d_{n+1}$ in the interval \eqref{eq79} so that none of the values $d_{k,n+1}$ ($1 \le k \le n+1$) coincide with the values $d_{j,m}$ ($1 \le j \le m \le n+1$).
  Whenever $\seq{\varepsilon_k}_{k=0}^{\infty}$ are themselves distinct, this is trivially possible.
  Indeed, since each $d_{k,n+1}$ is a differentiable function of $d_{n+1}$ for each $k$, with derivative $1/(n+1) + \varepsilon_k/((n+1)(n+2))$ and the $\varepsilon_k$ are distinct positive numbers, there are at most finitely many choices ($(n+1)(n+2)/2$) of $d_{n+1}$ for which $d_{k,n+1} = d_{j,n+1}$.
  Likewise, since each $d_{k,n+1}$ is a strictly increasing function of $d_{n+1}$, there are only finitely many choices of $d_{n+1}$ for which $d_{k,n+1} = d_{j,m}$ for some $1 \le j \le m \le n$ and $1 \le k \le n+1$.
  Since there are only finitely many such values but infinitely many choices for $d_{n+1}$ in the interval \eqref{eq79}, it is possible to choose all the $d_{k,n}$ distinct.
  In fact, choosing each $d_{n+1}$ in the interval \eqref{eq79} uniformly at random (when the $\seq{\varepsilon_k}_{k=0}^{\infty}$ are distinct) would result in distinct $d_{k,n}$ with probability $1$.
\end{proof}

These techniques somewhat resemble partial results on the same distinctness problem appearing in the Davidson, Marcoux, and Radjavi preprint (private communication) \cite{DMR} mentioned earlier.


\section{Anderson's model and its generalization: obstructions and limitations}
\label{sec3a}

In view of the progress accomplished in \Cref{sec3}, it is natural to ask whether a clever choice of entries in the off-diagonal blocks of Anderson's model could always yield an arbitrary positive diagonal operator as a commutator of compact operators.
It turns out that this is not the case.
Indeed, the arithmetic growth of his block sizes presents an obstruction (see \Cref{cor:anderson-model-obstruction} and \Cref{thm:exponential-growth-block-size-required}).

In this section, we will concern ourselves with the limitations of Anderson's model by establishing necessary conditions on certain generalizations of his model.
Throughout this section we will express $C,Z$ in \term{block matrix form} with finite block sizes, that is,
\begin{equation}
  \label{E}
  C =
  \begin{pmatrix}
    C_1 &  A_1  &*  &\dots          \\
    B_1 & C_2&A_2   &\ddots           \\
    * &B_2 &C_3 &\ddots            \\
    \vdots &\ddots &\ddots  &\ddots
  \end{pmatrix},
  \qquad
  Z =
  \begin{pmatrix}
    Z_1 &  X_1  &*  &\dots          \\
    Y_1 & Z_2&X_2   &\ddots           \\
    * &Y_2 &Z_3 &\ddots            \\
    \vdots &\ddots &\ddots  &\ddots
  \end{pmatrix},
\end{equation}
with the central blocks $C_n, Z_n$ of size $k_n\times k_n$.
We will be especially interested in the case when $C$ (or alternatively $Z$) is in \term{block tridiagonal form};
that is, when the blocks of $C$ (or alternatively $Z$) specified by asterisks in \eqref{E} are zero.
This is pertinent for two reasons:
\begin{enumerate*} 
\item in Anderson's model, both operators were in block tridiagonal form;
\item in \Cref{sec2}, we show how to place any operator in block tridiagonal form with respect to some basis.
\end{enumerate*}

Our constraints will often be phrased in terms of norm conditions on the \emph{rectangular} blocks $A_n$ and $B_n$.
Of course, $\norm{A_n}$ is the usual operator norm, but we will also use the notation $\norm{A_n}_1$ to refer to the trace norm $\trace((A_n^{*} A_n))^{1/2})$.
We note the submultiplicative inequality $\norm{A_n B_n}_1 \le \norm{A_n}_1 \norm{B_n}$, or $\norm{A_n B_n}_1 \le \norm{B_n}_1 \norm{A_n}$, is standard (see \cite[p.~454, 7.3.P16]{HJ-2013-MA}).
Of course, this also generalizes to $\norm{A_n B_n}_1 \le \norm{A_n}_p \norm{B_n}_q$ whenever $1/p + 1/q = 1$, where $\norm{A_n}_p := \trace((A_n^{*}A_n)^{p/2})^{1/p}$ denotes the Schatten $p$-norm.
In addition, we will also be concerned with conditions involving the block size sequence $\seq{k_n}$.

Our most general result (\Cref{T.1}) provides a trace restriction on the rectangular blocks $A_n,B_n$ when either $C$ or $Z$ is in block tridiagonal form, even in the case when both $C,Z$ are bounded (not necessarily compact) operators.

\begin{theorem}
  \label{T.1}
  Let $C, Z \in \BH$ and $T = \comm{C}{Z}$.
  Let $\seq{e_n}$ be any basis in which $C,Z$ have block matrix form \eqref{E}, i.e.,
  \begin{equation}
    \label{TDF2}
    C =
    \begin{pmatrix}
      C_1 &  A_1  &*   & \cdots        \\
      B_1 & C_2&A_2  &  \ddots            \\
      * &B_2 &C_3 & \ddots             \\
      \vdots &\ddots &\ddots  &\ddots  \\
    \end{pmatrix}
    \quad \text{and} \quad
    Z =
    \begin{pmatrix}
      Z_1 &  X_1  &*   & \cdots        \\
      Y_1 & Z_2&X_2  &  \ddots            \\
      * &Y_2 &Z_3 & \ddots             \\
      \vdots &\ddots &\ddots  &\ddots  \\
    \end{pmatrix}
  \end{equation}
  both with the central blocks of sizes $k_n \times k_n$ where $\seq{k_n}$ is a sequence of positive integers.
  Let $Q_n$ denotes the projection onto $\bigvee_{j=1}^{k_1+\cdots+k_n} e_j$.
  If either $C$ or $Z$ is block tridiagonal in \eqref{TDF2}, then
  \begin{equation}
    \label{3.11}
    \abs{\trace(Q_n T Q_n)} \le \norm{A_n Y_n - X_n B_n}_1 \le (\norm{A_n}_1 + \norm{B_n}_1) \norm{Z} \le k_n(\norm{A_n} + \norm{B_n}) \norm{Z}.
  \end{equation}
\end{theorem}

\begin{proof}
  Let $P_n$ denote the projection onto $\bigvee_{j=k_{n-1}+1}^{k_n} e_j$, where we set $k_0 = 0$ for convenience.
  Then clearly $Q_n = P_1 + \cdots + P_n$.
  Let $T_n$ denote the compression of $T = \comm{C}{Z}$ to $P_n\Hil$.
  Since either $C$ or $Z$ is block tridiagonal, a straightforward computation produces
  \begin{equation*}
    T_1 = A_1 Y_1 - X_1 B_1 + \comm{C_1}{Z_1},
  \end{equation*}
  and for all $n \ge 1$,
  \begin{equation*}
    T_{n+1} = A_{n+1} Y_{n+1} - X_{n+1} B_{n+1} - A_n Y_n + X_n B_n + \comm{C_{n+1}}{Z_{n+1}}.
  \end{equation*}
  Since $\comm{C_n}{Z_n}$ is a commutator of finite matrices, it has trace zero, and therefore
  \begin{equation*}
    \trace(T_1) = \trace(A_1 Y_1 - X_1 B_1),
  \end{equation*}
  and for all $n \ge 1$,
  \begin{equation*}
    \trace(T_{n+1}) = \trace(A_{n+1} Y_{n+1} - X_{n+1} B_{n+1} - A_n Y_n + X_n B_n).
  \end{equation*}
  Consequently, the following sum is telescoping, and hence
  \begin{equation}
    \label{eq:telescoping-partial-trace}
    \trace(Q_nTQ_n) = \sum_{j=1}^n \trace(T_j) = \trace(A_n Y_n - X_n B_n).
  \end{equation}
  Therefore,
  \begin{align*}
    \abs{\trace(Q_n T Q_n)} &= \abs{\trace(A_n Y_n - X_n B_n)} \\
                            &\le \norm{A_n Y_n - X_n B_n}_1 \\
                            &\le \norm{A_n}_1 \norm{Y_n} + \norm{B_n}_1 \norm{X_n} \\
                            &\le (\norm{A_n}_1 + \norm{B_n}_1) \norm{Z} \\
                            &\le k_n(\norm{A_n} + \norm{B_n}) \norm{Z}. \qedhere
  \end{align*}
\end{proof}

\begin{remar}
  Other choices of H\"{o}lder's inequality yield similar information concerning Schatten $p$-norms which we will not pursue here.
\end{remar}

When $T = \comm{C}{Z}$ is either trace-class or positive, the sequence $\trace(Q_n T Q_n)$ appearing in \eqref{3.11} has a (possibly infinite) limit, which quickly leads to the following corollary.
Here, additional information can be extracted when $Z$ is compact.

\begin{corollary}
  Let $C,Z \in \BH$ and $T = \comm{C}{Z}$, and let $\seq{e_n}$, $\seq{k_n}$ and $\seq{Q_n}$ be as in \Cref{T.1}, with either $C$ or $Z$ in block tridiagonal form.
  If $T \neq 0$ is trace-class or positive, then
  \begin{equation}
    \label{3.11-trace-class-positive}
    0 < \frac{\abs{\trace(T)}}{\norm{Z}} \le \liminf_{n \to \infty} (\norm{A_n}_1 + \norm{B_n}_1) \le \liminf_{n \to \infty} k_n(\norm{A_n} + \norm{B_n}).
  \end{equation}
  Moreover, if $Z$ is compact, then these limits inferior must in fact be infinite.
\end{corollary}

\begin{proof}
  To prove \eqref{3.11-trace-class-positive}, apply \Cref{T.1}, divide by $\norm{Z}$ throughout in \eqref{3.11} and take the limit inferior.
  Meanwhile, note that since $T \neq 0$ is trace-class or positive, then the trace is well-defined (including the trace being infinite), and hence
  \begin{equation*}
    0 < \frac{\abs{\trace(T)}}{\norm{Z}} = \lim_{n \to \infty} \frac{\abs{\trace(Q_n T Q_n)}}{\norm{Z}} = \liminf_{n \to \infty} \frac{\abs{\trace(Q_n T Q_n)}}{\norm{Z}}.
  \end{equation*}

  The claim concerning compact $Z$ takes slightly more work.
  Note that
  \begin{equation*}
    \abs{\trace(Q_nTQ_n)} \le \norm{A_n Y_n - X_n B_n}_1 \le \norm{A_n}_1 \norm{Y_n} + \norm{B_n}_1 \norm{X_n} \le (\norm{A_n}_1 + \norm{B_n}_1)(\norm{X_n} + \norm{Y_n}),
  \end{equation*}
  Since $T \neq 0$ is either trace-class or positive, the first inequality of \eqref{3.11} guarantees $\norm{X_n} + \norm{Y_n}$ is strictly positive for all sufficiently large $n$.
  Therefore, 
  dividing by $\norm{X_n} + \norm{Y_n}$ in the previous display yields
  \begin{equation*}
    \frac{\abs{\trace(Q_n T Q_n)}}{\norm{X_n} + \norm{Y_n}} \le \norm{A_n}_1 + \norm{B_n}_1.
  \end{equation*}
  Notice that $\norm{X_n},\norm{Y_n} \to 0$ since $Z$ is compact.
  Taking the limit inferior in the previous display and noting that the left-hand side approaches infinity proves the claim.
\end{proof}

Of course, in the case of interest when both $C,Z$ are compact, this becomes:

\begin{corollary}
  \label{cor:CZ-compact-norm-size-constraint}
  Let $C,Z \in \KH$ and $T = \comm{C}{Z}$, and let $\seq{e_n}$, $\seq{k_n}$ and $\seq{Q_n}$ be as in \Cref{T.1}, with either $C$ or $Z$ in block tridiagonal form.
  If $T \neq 0$ is trace-class or positive, then
  \begin{equation*}
    \lim_{n \to \infty} k_n(\norm{A_n} + \norm{B_n}) = \infty = \lim_{n \to \infty} k_n(\norm{X_n} + \norm{Y_n})
  \end{equation*}
  Moreover, the same holds for arbitrary $T$ whenever $\liminf_{n \to \infty} \abs{\trace(Q_n T Q_n)} > 0$.
\end{corollary}

The next theorem presents an obstruction (explicitly realized in \Cref{ex:anderson-model-failure} and the proof of \Cref{lem:exponential-growth-block-size}) to obtaining a positive diagonal operator as a commutator of a compact operator and a bounded operator, at least one of which is block tridiagonal.

\begin{theorem}
  \label{thm:k_n-average-diagonal-obstruction}
  Let $\seq{k_n}$ be a sequence of positive integers and $\seq{d_n}$ be a nonnegative sequence for which
  \begin{equation}
    \label{eq34}
    \limsup_{n \to \infty} \frac{1}{k_n}\left(d_1 +  \cdots + d_{k_1+\dots+k_n}\right) > 0,
  \end{equation}
  and let $D = \diag\seq{d_n}$ relative to the basis $\seq{e_n}$.
  Then $D$ is not a commutator of a compact operator $C$ and a bounded operator $Z$, both in block matrix form \eqref{E} relative to this basis with central block sizes $k_n \times k_n$, where at least one of $C,Z$ is block tridiagonal.
\end{theorem}

\begin{proof}
  Suppose that $D$ is a commutator of operators $C,Z$ of the specified form and at least one of these is in block tridiagonal form.
  We will show that $C$ cannot be compact.
  
  Let $Q_n$ be the projection onto $\bigvee_{j=1}^{k_1 + \cdots + k_n} e_j$.
  Then by \Cref{T.1},
  \begin{equation*}
    d_1 + \cdots + d_{k_1 + \cdots + k_n} = \trace(Q_n D Q_n) \le k_n (\norm{A_n} + \norm{B_n}) \norm{Z}.
  \end{equation*}
  Dividing by $k_n \norm{Z}$ and taking the limit superior, by hypothesis we obtain
  \begin{equation*}
    0 < \limsup_{n \to \infty} \frac{1}{k_n \norm{Z}}\left(d_1 +  \cdots + d_{k_1+\dots+k_n}\right) \le \limsup_{n \to \infty} (\norm{A_n} + \norm{B_n}). 
  \end{equation*}
  Therefore at least one of $\norm{A_n},\norm{B_n}$ does not converge to zero, and hence $C$ cannot be compact.
\end{proof}

\begin{remar}
  Because it is always possible to rearrange $\seq{d_n}$ when $\lim d_n=0$ so that the limit superior in \eqref{eq34} is 0, it becomes clear that \Cref{thm:k_n-average-diagonal-obstruction} is dependent on the order of $\seq{d_n}$, even for a fixed sequence of block sizes $\seq{k_n}$.

  It is also useful to know that condition \eqref{eq34} holds for some rearrangement of $\seq{d_n}$ if and only if it holds for its decreasing rearrangement.
  Indeed, if $\seq{d_n^{\star}}$ denotes the decreasing rearrangement of $\seq{d_n}$ then clearly for each $n \ge 1$,
  \begin{equation*}
    \frac{1}{k_n}\left(d_1 +  \cdots + d_{k_1+\dots+k_n}\right)\le \frac{1}{k_n}\left(d_1^{\star} +  \cdots + d_{k_1+\dots+k_n}^{\star}\right).
  \end{equation*}
\end{remar}

The previous theorem has the following immediate corollary concerning commutators obtained via the Anderson model detailed in \eqref{eq6}--\eqref{eq5}.

\begin{corollary}
  \label{cor:anderson-model-obstruction}
  Suppose that $\seq{d_n}$ is a sequence of strictly positive numbers such that
  \begin{equation}
    \label{eq33}
    \limsup \frac{1}{n}\left(d_1 +  \cdots + d_{n(n+1)/2}\right)> 0,
  \end{equation}
  and $D = \diag\seq{d_n}$.
  Then $D$ cannot be obtained using the Anderson model.
  That is to say, $D$ is not a commutator of a compact  operator $C$ and a bounded operator $Z$ both of the form \eqref{eq6}--\eqref{eq5} with central blocks sizes  $n\times n$.
  The same holds even more generally when $A_n$, $B_n$, $X_n$, and $Y_n$ are possibly full rectangular matrices, or even if the diagonal blocks of $C$ are nonzero.
\end{corollary}

\begin{example}
  \label{ex:anderson-model-failure}
  There are many sequences satisfying \eqref{eq33}, for instance, $d_n = {1}/{\sqrt{n}}$ (whose partial sums satisfy $\sum^{n(n+1)/2}_{j=1} d_j = \Theta(n)$ in Bachmann--Landau notation, i.e., $f(n) = \Theta(g(n))$ means $f(n) = O(g(n))$ and $g(n) = O(f(n))$) is an example where Anderson's model fails by \Cref{cor:anderson-model-obstruction}.
  In \Cref{thm:exponential-growth-block-size-required}, we will see that this is a consequence of the arithmetic (hence subexponential) growth of the block sizes in Anderson's model.
\end{example}

\begin{remark}
  \label{rem:exponential-growth-obstruction-disappears}
  We note that when the block sizes $\seq{k_n}$ grow exponentially (i.e., $k_n \ge \Theta(\rho^n)$ for some $\rho > 1$), then the sequence $\seq{(k_1 + \cdots + k_n)/k_n}$ is bounded above.
  Then, for any nonnegative sequence $\seq{d_n}$ converging to zero, it is well known that its Ces\`{a}ro means tend to zero, and hence
  \begin{align*}
    \limsup_{n \to \infty} \frac{1}{k_n}\left(d_1 +  \cdots + d_{k_1+\dots+k_n}\right)
    &= \limsup_{n \to \infty} \frac{k_1 + \cdots + k_n}{k_n} \cdot \frac{d_1 +  \cdots + d_{k_1+\cdots+k_n}}{k_1 + \cdots + k_n} \\
    &= \limsup_{n \to \infty} \frac{k_1 + \cdots + k_n}{k_n} \cdot \lim_{n \to \infty} \frac{d_1 +  \cdots + d_{k_1+\cdots+k_n}}{k_1 + \cdots + k_n} = 0,
  \end{align*}
  since $(k_1 + \cdots + k_n)/k_n$ is bounded above.
\end{remark}

The next two lemmas show that for any sequence $\seq{k_n}$ which does not grow sufficiently rapidly, there is a strictly positive nonincreasing sequence $\seq{d_n}$ satisfying condition \eqref{eq33}.
Moreover, this growth condition implies the sequence $\seq{k_n}$ grows \emph{at least} exponentially.

\begin{lemma}
  \label{lem:exponential-growth-block-size}
  Let $\seq{k_n}_{n=1}^{\infty}$ be a sequence of positive integers, and let $\seq{s_n}_{n=1}^{\infty}$ denote its sequence of partial sums.
  If for every nonincreasing positive sequence $\seq{d_n}_{n=1}^{\infty}$ converging to zero one has
  \begin{equation*}
    \limsup_{n \to \infty} \frac{1}{k_n} \sum_{j=1}^{s_n} d_j < \infty,
  \end{equation*}
  then $\liminf_{n \to \infty} k_n/s_n > 0$.
\end{lemma}

\begin{proof}
  We will prove the contrapositive.
  Suppose that $\liminf_{n \to \infty} k_n/s_n = 0$.
  Then $\limsup_{n \to \infty} s_n/k_n = \infty$.
  So there is a strictly increasing sequence of positive integers $\seq{n_l}_{l=1}^{\infty}$ such that for all $l \ge 1$, $s_{n_l}/k_{n_l} > l$.

  For convenience of notation, set $n_0 := 0$ and $s_0 := 0$.
  Then define a positive nonincreasing sequence $\seq{d_j}_{j=1}^{\infty}$ converging to zero by $d_j := (\log (l+1))/l$ for $s_{n_{l-1}} < j \le s_{n_l}$.
  Then we clearly have for all $l$,
  \begin{equation*}
    \frac{1}{s_{n_l}} \sum_{j=1}^{s_{n_l}} d_j \ge \frac{\log (l+1)}{l}.
  \end{equation*}
  Consequently,
  \begin{equation*}
    \frac{1}{k_{n_l}} \sum_{j=1}^{s_{n_l}} d_j = \frac{s_{n_l}}{k_{n_l}} \cdot \frac{1}{s_{n_l}} \sum_{j=1}^{s_{n_l}} d_j \ge l \cdot \frac{\log (l+1)}{l} = \log (l+1).
  \end{equation*}
  Hence $\limsup_{n \to \infty} \frac{1}{k_n} \sum_{j=1}^{s_n} d_j = \infty$
\end{proof}

\begin{lemma}
  \label{lem:liminf-ratio-positive-theta-exponential}
  Let $\seq{k_n}_{n=1}^{\infty}$ be a sequence of positive integers, and let $\seq{s_n}_{n=1}^{\infty}$ denote its sequence of partial sums.
  If $\liminf_{n \to \infty} k_n/s_n > 0$, then $k_n = \Omega(\rho^n)$ for some $\rho > 1$ (Bachmann--Landau notation, $f(n) = \Omega(g(n))$ is $g(n) = O(f(n))$).
  In fact, for any $0 < \delta < \liminf_{n \to \infty} k_n/s_n$, we may choose $\rho = 1 + \delta$.
\end{lemma}

\begin{proof}
  Take any $0 < \delta < \liminf_{n \to \infty} k_n/s_n$, and let $\rho = 1 + \delta$.
  Then there is some $N \ge 1$ such that $k_n/s_n > \delta$ for all $n \ge N$.\footnotemark{}
  Then for any $n \ge N$,
  \begin{equation*}
    s_{n+1} = k_{n+1} + s_n > \delta s_{n+1} + s_n > \rho s_n.
  \end{equation*}
  Now a simple induction shows, for all $n \ge N$, $s_n \ge \rho^{n-N} s_N$, and hence $k_n \ge \delta \rho^{n-N} s_N$, so $k_n = \Omega(\rho^n)$.
\end{proof}

\footnotetext{%
  Note that this is essentially a discrete differential inequality since
  \begin{equation*}
    \frac{s_n - s_{n-1}}{n - (n-1)} = k_n > \delta s_n.
  \end{equation*}
}

\begin{theorem}
  \label{thm:exponential-growth-block-size-required}
  Let $\seq{k_n}$ be a sequence of positive integers.
  Suppose that for every strictly positive nonincreasing sequence $\seq{d_n}$ converging to zero, $D = \diag\seq{d_n}$ (relative to a basis $\seq{e_n}_{n=1}^{\infty}$) is a commutator $\comm{C}{Z}$ where $C$ is compact and either $C$ or $Z$ has block tridiagonal form (relative to $\seq{e_n}_{n=1}^{\infty}$) with central block sizes $k_n$.
  Then $\seq{k_n}$ has at least exponential growth: $k_n = \Omega(\rho^n)$ for some $\rho > 1$.
\end{theorem}

\begin{proof}
  Suppose that $\seq{k_n}$ is such a sequence as specified in the theorem.
  Then by hypothesis, for any nonincreasing positive sequence $\seq{d_n}$ converging to zero, $\diag\seq{d_n} = \comm{C}{Z}$ for some compact operator $C$ such that either $C$ or $Z$ is block tridiagonal.
  Then the contrapositive of \Cref{thm:k_n-average-diagonal-obstruction} guarantees
  \begin{equation*}
    \limsup_{n \to \infty} \frac{1}{k_n}\left(d_1 +  \cdots + d_{k_1+\dots+k_n}\right) = 0.
  \end{equation*}
  Since this holds for all positive nonincreasing sequences converging to zero, \Cref{lem:exponential-growth-block-size,lem:liminf-ratio-positive-theta-exponential} prove the desired claim.
\end{proof}

The exponential growth requirement from \Cref{thm:exponential-growth-block-size-required} together with the disappearance of the obstruction \eqref{eq34} when $\seq{k_n}$ grows exponentially (see \Cref{rem:exponential-growth-obstruction-disappears}) suggest the possibility that every strictly positive diagonal operator is a commutator of compact operators via block tridiagonal operators with exponentially growing block sizes.
We will investigate this idea further in \Cref{sec2}, but given the relative success of Anderson's model, it is natural to wonder whether this is possible with an exponential version of Anderson's model (EAM) displayed in \eqref{blockseam} below.
Namely, let $\seq{k_n}$ be a sequence of positive integers satisfying $k_{n+1}\ge q k_n$ for some $q>1$ (which includes the case  equivalent to $a\rho^n$ for integers $a>0,\rho>1$), and let
\begin{equation}
  \label{blockseam}
  \begin{alignedat}{2}
    A_n &=
    \begin{pmatrix}
      a_{1,n} & 0 & \cdots   & 0 & 0 & \cdots & 0 \\
      0 & a_{2,n}  & \ddots  & \vdots & \vdots & \cdots & \vdots \\
      \vdots &\ddots & \ddots & 0 & \vdots & \cdots & \vdots \\
      0 & \cdots & 0 & a_{k_n,n} & 0 & \cdots & 0
    \end{pmatrix},\quad
    &X_n =&
    \begin{pmatrix}
      0 & x_{1,n} & 0  & \cdots & 0 & 0 & \cdots & 0 \\
      \vdots & 0 & x_{2,n} & \ddots & \vdots & \vdots & \cdots & \vdots \\
      \vdots & \vdots & \ddots & \ddots & 0 & \vdots & \cdots & \vdots \\
      0 & 0 & \cdots & 0 & x_{k_n,n} & 0 & \cdots & 0
    \end{pmatrix},\\
    B_n &=
    \begin{pmatrix}
      0 & \cdots &\cdots & 0 \\
      -b_{1,n} & 0 & \cdots & 0 \\
      0 &-b_{2,n}  & \ddots & \vdots \\
      \vdots & \ddots & \ddots & 0 \\
      0 & \cdots & 0 & -b_{k_n,n} \\
      0 & \cdots &\cdots & 0 \\
      \vdots & \vdots & \vdots & \vdots \\
      0 & \cdots &\cdots & 0 \\
    \end{pmatrix},\quad
    &Y_n =&
    \begin{pmatrix}
      y_{1,n} & 0  & \cdots  & 0 \\
      0 & y_{2,n} & \ddots & \vdots \\
      \vdots & \ddots & \ddots & 0 \\
      0 & \cdots & 0 & y_{k_n,n}\\
      0 & \cdots & \cdots & 0 \\
      \vdots & \vdots & \vdots & \vdots \\
      0 & \cdots &\cdots & 0 \\
    \end{pmatrix},
  \end{alignedat}
\end{equation}
where these matrices are of the sizes $k_n\times k_{n+1}$ ($A_n$ and $X_n$) and $k_{n+1}\times k_{n}$ ($B_n$ and $Y_n$).
The model (EAM) is obtained by defining  $C,Z$ as in \eqref{eq6}, but with the central blocks of size $k_n\times k_n$.
It is natural to ask whether the commutator $CZ-ZC$ can yield every compact diagonal operator.
We show that the answer remains: No!

\begin{theorem}
  \label{tEAM}
  If the (EAM) yields every compact diagonal operator as a commutator of compact operators, then so also does the (AM).
\end{theorem}

\Cref{cor:anderson-model-obstruction} and \Cref{ex:anderson-model-failure} show that the (AM) cannot accomplish this goal  so it follows by contrapositive that neither can the (EAM).

\begin{corollary}
  \label{NOTtEAM} Not all strictly positive diagonal operators can be represented as commutators of compact operators using (EAM).
\end{corollary}

\begin{proof}[Proof of \Cref{tEAM}]
  Suppose that $D'=\diag (D'_n)$ is a compact operator, with each block $D_n'$ being a diagonal $n\times n$ matrix.
  For each $n\in\mathbb{N}$, let $D_n$ be a $k_n\times k_n$ diagonal matrix such that the upper left $n\times n$ block is $D_n'$, i.e.,
  \begin{equation*}
    D_n = \begin{pmatrix} D_n' &0\\0& \ast\end{pmatrix}.
  \end{equation*}
  Of course, if $k_n=n$ then $D_n=D_n'$, and we notice that this can happen only for $k_1 = 1$.
  Since (EAM) by hypothesis can represent $D=\diag (D_n)$ as a commutator of compact operators $C,Z$ as in \eqref{eqD_N}, but with  central blocks of size $k_n\times k_n$, and with blocks satisfying for $n \ge 1$,
  \begin{gather}
    A_nX_{n+1}=X_nA_{n+1}\quad\mbox{ and }\quad B_{n+1}Y_n=Y_{n+1}B_n,\label{offdiag}\\
    D_{n}=A_{n}Y_{n}-X_{n}B_{n}+B_{n-1}X_{n-1}-Y_{n-1}A_{n-1},\label{eq13}
  \end{gather}
  taking $A_0=B_0=X_0=Y_0=0$.
  Furthermore, the blocks in (EAM) are given by \eqref{blockseam} and it is not hard to see that each can be written in the form
  \begin{equation}
    \label{matrices}
    A_n=\begin{pmatrix} A_n' & 0\\ \ast & \ast\end{pmatrix},
    X_n=\begin{pmatrix} X_n' & 0\\ 0 & \ast\end{pmatrix},
    B_n=\begin{pmatrix} B_n' & 0\\ 0 & \ast\end{pmatrix},
    Y_n=\begin{pmatrix} Y_n' & \ast\\ 0 & \ast\end{pmatrix},
  \end{equation}
  where the blocks $A_n',X_n',B_n',Y_n'$ are given by \eqref{eq5}.
  (Once again, if $k_1=1$, then $A_1=A_1'$.)
  Substituting \eqref{matrices} into \eqref{offdiag}--\eqref{eq13} we obtain
  \begin{equation}
    \label{matrixeqs}
    \begin{gathered}
      \begin{pmatrix} A_{n}'X'_{n+1} & \ast\\ \ast & \ast\end{pmatrix}=\begin{pmatrix} X_n'A_{n+1}' & \ast\\ \ast & \ast\end{pmatrix}\quad\mbox{ and }\quad
      \begin{pmatrix} B'_{n+1}Y'_{n} & \ast\\ \ast & \ast\end{pmatrix}=\begin{pmatrix} Y'_{n+1}B'_n & \ast\\ \ast & \ast\end{pmatrix}\\
      \begin{pmatrix} D_n' &0\\0& \ast\end{pmatrix}=\begin{pmatrix} A'_{n}Y'_{n} & \ast\\ \ast & \ast\end{pmatrix}-\begin{pmatrix} X'_{n}B'_{n} & \ast\\ \ast & \ast\end{pmatrix}+\begin{pmatrix} B'_{n-1}X'_{n-1} & \ast\\ \ast & \ast\end{pmatrix}-\begin{pmatrix} Y'_{n-1}A'_{n-1} & \ast\\ \ast & \ast\end{pmatrix}.
    \end{gathered}
  \end{equation}
  The only calculation that may not be immediately obvious concerns the product $Y_{n-1}A_{n-1}$ which is a diagonal matrix, and its upper left corner is indeed $Y'_{n-1}A'_{n-1}$.
  It follows from \eqref{matrixeqs} that the smaller blocks $A'_n,X'_n,B'_n,Y'_n$ satisfy conditions \eqref{offdiag}--\eqref{eq13}.
  Consequently, $D'=\comm{C'}{Z'}$, where $C',Z'$ are of the form \eqref{eq6} but with blocks $A_n',B_n',X_n',Y_n'$.
  In other words $D'$ is a commutator of compact operators given by (AM).
\end{proof}

\section{Universal Block Tridiagonalization and the generalized Anderson model (GAM)}
\label{sec2}

In this section we will consider the problem of general matrix sparsification under a special unitary equivalence.
One approach to the general problem for arbitrary $\BH$-operators is staircase forms given in \cite{Wei-1980-IEOT}, which was an attempt to extend to all operators the well-known tridiagonalization of selfadjoint operators that possess a cyclic vector.
This was essential to the solution in \cite{Wei-1980-IEOT} to \cite[Problems 3 and 3$'$]{PT-1971-MMJ} and led to \cite{PPW-2020-TmloVLOt}, a precursor to this paper.
Some of the material presented in this section appears in \cite{PPW-2020-TmloVLOt} but we include it here for completeness, enhanced and somewhat reformulated but with less detail (\Cref{T1}--\Cref{rem1}).

The following general staircase form \eqref{E0} is obtained in \cite{Wei-1980-IEOT} by considering the free semigroup on the two generators $C,C^*$ with any basis $\seq{e_n}$ to generate a new basis via applying Gram--Schmidt to $e_1$, $Ce_1$, $C^*e_1$, $e_2$, $C^2e_1$, $C^*Ce_1$, $e_3$, $CC^*e_1$, ${C^*}^2e_1$, $e_4$, $Ce_2$, $C^*e_2$, $e_5$, $\dots$, when this list is linearly independent.%
\footnote{%
  When the list is linearly dependent, an induction argument is necessary.
}
This list consists of $\mathcal{W}(C,C^*)$ (all words in $C,C^*$, including the empty word $I$) applied to all elements of $\seq{e_n}$ and arranged in this special order:
first list $e_1$ followed by $C,C^*$ applied to $e_1$ to obtain the first three vectors;
then list $e_2$, followed by $C,C^*$ applied to the second vector in this list, namely $Ce_1$;
and so on.
By a \term{staircase form}, we mean a matrix and two strictly increasing sequences of positive integers $\seq{r_1(n)},\seq{r_2(n)}$, called the \term{column and row support lengths}, for which the matrix entries are zero in the positions $(i,j)$ whenever $i > r_1(j)$ or $j > r_2(i)$.

\begin{theorem}[\protect{\cite[Theorem 2--Corollary 3, combined]{Wei-1980-IEOT}}]
  \label{t1}
  Let $C \in \BH$ be arbitrary and let $\seq{e_n}$ be any orthonormal basis of $\Hil$.
  There is an orthonormal basis $\seq{f_n}$ is derived from $\seq{e_n}$ (via the method in the above paragraph) such that $f_1 = e_1$ and $e_n \in \vee_{k=1}^{3n} f_k$ for all $n \in \mathbb{N}$, and moreover, the matrix of $C$ with respect to $\seq{f_n}$ takes the staircase form
  \begin{equation}
    \label{E0}
    C =
    \begin{pmatrix}
      *&*&*&0&\cdots\\
      *&*&*&*&*&*&0& \cdots\\
      *&*&\cdots\\
      0&*&\\
      0&*&\\
      0&*&\\
      \vdots&0&
    \end{pmatrix},
  \end{equation}
  where this matrix has row and column support lengths $\seq{3n}_{n=1}^{\infty}$.

  In addition, if $S_1, \ldots, S_N$ is any finite collection of selfadjoint operators, then there is an orthonormal basis $\seq{f_n}$ for which $f_1 = e_1$ and $e_n\in\vee_{k=1}^{1+(n-1)(N+1)} f_k$, and with respect to which $S_k$ takes the form \eqref{E0}  with $\seq{3n}_{n=1}^{\infty}$ replaced by $\seq{(N+1)n}_{n=1}^{\infty}$.
  
  When $S_1, \ldots, S_N$ are not necessarily selfadjoint, the row and column support lengths are given by $\seq{(2N+1)n}_{n=1}^{\infty}$, and the inclusion relationships are $e_n\in\vee_{k=1}^{1+(n-1)(2N+1)} f_k$.
\end{theorem}

\begin{definition}
  \label{def4.2}
  We call the $*$-entries in \eqref{E0} the support entries.
  Albeit some can also be zero as we shall see in \Cref{T2,t3aa}.
  So for instance, for $N=1$ we obtained support row and column sizes $\seq{3n}$, which technically have fewer support entries because of the initial zeros (starting for instance with the 4th row and column which has 11 asterisks and not 12), but for simplicity we don't mention those.
\end{definition}

\Cref{t1} gives a striking universal staircase form (with step sequence $\seq{3n}$) for an arbitrary operator and universal simultaneous staircase forms with larger stairs for arbitrary finite collections.
One can slightly strengthen \Cref{t1} to have column lengths $\seq{3n-1}$ (see \cite[Remark 20.3]{PPW-2020-TmloVLOt}).

The staircase form \eqref{E0} will allow us to represent a matrix in universal block tridiagonal form as in \eqref{F} below:
\begin{equation}
  \label{F}
  C =
  \begin{pmatrix}
    C_1 &  A_1  &0  &\dots          \\
    B_1 & C_2&A_2   &\ddots           \\
    0 &B_2 &C_3 &\ddots            \\
    \vdots &\ddots &\ddots  &\ddots
  \end{pmatrix}.
\end{equation}
The following theorem gives canonical dimensions for the blocks in \eqref{F} in order that they cover all the support entries.
And since there will be no change of basis, we retain the \Cref{t1} condition that each $e_n\in \vee_{k=1}^{3n} f_k$.

\begin{theorem}
  \label{T1}
  For all $C \in \BH$, the block tridiagonal partition of the matrix of $C$ induced by \eqref{E0} (i.e.,  in the derived basis $\seq{f_n}$ from \Cref{t1}) is given by \eqref{F},
  where the central blocks $C_n$ have block size sequence $\seq{k_n}$ given by $k_1 = 1$ and $k_n = 2(3^{n-2})$ for $n > 1$.

  More generally, the necessary and sufficient conditions that \eqref{F} cover the \eqref{E0} staircase support entries are: $k_1$ is chosen arbitrarily and $k_{n+1} \ge 2 ( k_1 + k_2 + \dots + k_n )$.
\end{theorem}

\begin{proof}
  Let $\seq{s_n}$ denote the sequence of partial sums of the block size sequence $\seq{k_n}$.
  Note that an entry in the $(i,j)$-position of $C$ is a support entry with respect to the basis $\seq{f_n}$ if and only if $i \le 3j$ and $j \le 3i$.
  The recursive condition $k_{n+1} \ge 2 s_n$ is clearly necessary since the $(s_n,3s_n)$-position is a support entry of $C$, and hence $3s_n \le s_{n+1}$ if the block tridiagonal form covers this support entry.

  For the sufficiency, suppose that $k_{n+1} \ge 2s_n$ for all $n \in \mathbb{N}$, so then $3s_n \le s_{n+1}$.
  Suppose the $(i,j)$-position is a support entry of $C$.
  Then there is some $n \in \mathbb{N}$ for which $s_{n-1} < i \le s_n$ (where we let $s_0 = 0$).
  Then $j \le 3 s_n \le s_{n+1}$.
  Similarly, for some $m \in \mathbb{N}$, $s_{m-1} < j \le s_m$ and hence $i \le s_{m+1}$.
  Hence the $(i,j)$-position is covered by the block tridiagonal form \eqref{F}.

  Finally, $k_n = 2(3^{n-2})$ for $n > 1$ is obtained by letting $k_1 = 1$ and using equality in the recursive formula.
\end{proof}

A similar result holds for more than one matrix.
For example, for 3 matrices we have the following consequence of the staircase form for the case $N=3$ of \Cref{t1}.
This is explained further in \Cref{rem1}.
And then a modified argument from the above paragraph shows the blocks cover the support entries of the larger staircase forms.

\begin{theorem}
  \label{t11}
  For any $N$ operators in $\BH$, the block tridiagonal partition of their matrices induced by the analog of \eqref{E0} for $N$ (i.e.,  in the new basis $\seq{f_n}$) is given by \eqref{F}.
  For each of them simultaneously the central blocks have block size sequence $\seq{k_n}$ given by $k_1 = 1$ and $k_n = 2N((2N+1)^{n-2})$ for $n > 1$.

  In particular, for $3$ operators $k_n = 6(7^{n-2})$, and for $2$ operators $k_n = 4(5^{n-2})$ for $n > 1$.
\end{theorem}

A brief explanation for this theorem is:

\begin{remark}
  \label{rem1}
  We mention that in the case of more than one operator, the orthonormal basis $\seq{f_n}$ of $\Hil$ --- with respect to which their matrices are in the block tridiagonal form of \Cref{t11} --- can be obtained in the following manner as in \cite[Theorem 2--Corollary 3, combined]{Wei-1980-IEOT}, at least when the list of vectors below is linearly independent.
  First we decompose each of the operators as $T=T_1+iT_2$, with $T_1$ and $T_2$ selfadjoint, and then we create a list of vectors.
  For instance, in the case of 3 operators $C=C_1+iC_2$, $Z=Z_1+iZ_2$, and $D=D_1+iD_2$, we use
  \begin{align*}
    &e_1, C_1e_1, C_2 e_1, Z_1e_1, Z_2e_1, D_1e_1,D_2e_1,\\
    &e_2, C_1 g_2, C_2g_2, Z_1g_2, Z_2g_2, D_1g_2,D_2g_2,\\
    &e_3, C_1 g_3, C_2g_3, Z_1g_3, Z_2g_3, D_1g_3,D_2g_3,\quad \mbox{etc.},
  \end{align*}
  where $g_2$ is the second vector, namely $C_1e_1$, $g_3$ is the third vector, namely $C_2e_1$, etc.
  Finally, Gram--Schmidt is applied to this list.
  This yields the staircase matrix analog of \eqref{E0} but with support row and column sizes $\seq{7n}$ replacing $\seq{3n}$.
  After Gram--Schmidt to obtain the analog staircase form, we then partition the staircase matrix into central, upper and lower blocks similarly as in the proof of \Cref{T1} based on the staircase matrix with row and column sizes $\seq{7n}$.
  And similarly for the case $N=2$ with analog staircase matrix row and column sizes $\seq{5n}$.
\end{remark}

\Cref{T1,t11} show that, when considering the equation $CZ-ZC=D$, there is no loss of generality in assuming that any two or all three operators $C,Z,D$ are in the block tridiagonal form  \eqref{F}.
We formulate this equivalence and leave the proof to the reader.

\begin{theorem}
  \label{equivthm}
  The following are equivalent:
  \begin{enumerate}
  \item $T$ is a commutator of compact operators.
  \item $T$ is a block tridiagonal operator of the form \eqref{F} relative to a basis $\seq{f_n}$, derived via \Cref{t1,T1,t11} from any basis $\seq{e_n}$, and $T$ is a commutator of compact matrices of the same block tridiagonal form relative to that same basis (i.e., all with central block sizes  $6 (7^{n-2})$ after $k_1=1$).
  \end{enumerate}
  For diagonal matrices, the following are equivalent:
  \begin{enumerate}[resume]
  \item The diagonal matrix $D$ (relative to a basis $\seq{e_n}$) is a commutator of compact operators.
  \item Relative to the basis $\seq{f_n}$, derived via \Cref{t1,T1,t11} from $\seq{e_n}$, $D$ is a block 5-diagonal operator of the form Equation \eqref{F}, and $D$ is a commutator of compact operators of the same block tridiagonal form relative to that same basis, all with central block sizes  $4 (5^{n-2})$ after $k_1=1$.
  \end{enumerate}
\end{theorem}

\begin{remark}
  \label{d vs d'}
  For positive operators, \Cref{equivthm} would include the additional sometimes useful inequalities that allow one to regulate the growth rate of the partial traces by selectively choosing various $T$ or $D$ targets.
  For $T$ or $D$ $=CZ-ZC$, one has the partial traces relative to the bases $\seq{e_n},\seq{f_n}$ related via $\sum_1^{\lfloor {n}/{7} \rfloor} (Te_i,e_i) \le \sum_1^n (Tf_i,f_i)$ and $\sum_1^{\lfloor {n}/{5} \rfloor} (De_i,e_i) \le \sum_1^n (Df_i,f_i)$.
  These can be deduced from the $e_n, f_n$ relations in \Cref{t1}, which remain in effect for \Cref{T1,t11}, applied as earlier in this paper to obtain bounds for trace computations of $D_1+\cdots + D_n$.

  Moreover, we have the following quantitative contraints on block norms: \eqref{eq4.7ageneral} and \eqref{eq4.7bgeneral} below, evolving from \eqref{3.11}, on the speed of convergence to $0$ of
  \begin{equation*}
    \norm{A_n}\norm{Y_n} + \norm{X_n}\norm{B_n} \le (\norm{A_n}+\norm{B_n})(\norm{Y_n} + \norm{X_n})
  \end{equation*}
  after block tridiagonalizing $C,Z$ in the basis $\seq{f_n}$.

  Let $D = \diag \seq{d_n}$, $d_n \downarrow 0$ be a diagonal operator in the basis $\seq{e_n}$.
  \Cref{T1,t11} allow us to derive a new basis $\seq{f_n}$ with respect to which $C,Z$ are simultaneously in block tridiagonal form with central block sizes $k_n = 4\cdot 5^{n-2}, k_1=1$:
  \begin{equation}
    \label{TDF4.6}
    C =
    \begin{pmatrix}
      C_1 &  A_1  &0   & \cdots        \\
      B_1 & C_2&A_2  &  \ddots            \\
      0 &B_2 &C_3 & \ddots             \\
      \vdots &\ddots &\ddots  &\ddots  \\
    \end{pmatrix}
    \quad and \quad
    Z =
    \begin{pmatrix}
      Z_1 & X _1  &0   & \cdots        \\
      Y_1 & Z_2&X_2  &  \ddots            \\
      0 &Y_2 &Z_3 & \ddots             \\
      \vdots &\ddots &\ddots  &\ddots  \\
    \end{pmatrix}.
  \end{equation}
  The diagonal entries $\seq{d_n'}$ of $D$ in the basis $\seq{f_n}$ are given by $d_n' = (Df_n,f_n)$.
  Let $D_n$ be the $n$-th diagonal block of $D = CZ-ZC$ with respect to the block 5-diagonal form on $D$ induced by \eqref{TDF4.6}.
  The inclusion relations governing $e_n,f_n$, along with the techniques from \Cref{sec3,sec3a}, guarantee
  \begin{align*}
    d_1 + d_2 + \cdots + d_{\lfloor ({k_1 + \cdots + k_n})/{5}\rfloor} &\le d_1' + d_2' + \cdots + d_{k_1 + \cdots + k_n}' = \trace(D_1+D_2+\dots+D_n)\\
                                                                       &=\trace(A_nY_n - X_nB_n) \le \norm{A_nY_n - X_nB_n}_1 \quad \text{(then applying H\"{o}lder)} \\
                                                                       &\le k_n(\norm{A_n}\norm{Y_n} + \norm{X_n}\norm{B_n}) \le k_n(\norm{A_n}+\norm{B_n})(\norm{Y_n} + \norm{X_n}).
  \end{align*}

  Similarly, for diagonal $D = \diag \seq{d_n}$, $d_n \downarrow 0$ in the basis $\seq{e_n}$, its derived basis $\seq{f_n}$ (via \Cref{t11}, $N=3$) simultaneously puts all three $D,C,Z$ in block tridiagonal form of central block sizes $k_n = 6\cdot 7^{n-2}, k_1=1$.
  Then
  \begin{align*}
    d_1 + d_2 + \cdots + d_{\lfloor ({k_1 + \cdots + k_n})/{7}\rfloor} &\le d_1' + d_2' + \cdots + d_{k_1 + \cdots + k_n}' = \trace(D_1+D_2+\dots+D_n)\\
                                                                       &=\trace(A_nY_n - X_nB_n) \le \norm{A_nY_n - X_nB_n}_1 \quad \text{(then applying H\"{o}lder)} \\
                                                                       &\le k_n(\norm{A_n}\norm{Y_n} + \norm{X_n}\norm{B_n}) \le k_n(\norm{A_n}+\norm{B_n})(\norm{Y_n} + \norm{X_n}).
  \end{align*}
  From these we obtain the respective quantitative constraints on the speed of convergence to $0$ of
  \begin{equation*}
    \norm{A_n}\norm{Y_n} + \norm{X_n}\norm{B_n} \le (\norm{A_n}+\norm{B_n})(\norm{Y_n} + \norm{X_n}):
  \end{equation*}
  For $N=2, i.e., C,Z$, we have
  \begin{equation}
    \label{eq4.7ageneral}
    \frac{1}{k_n} (d_1 + d_2 + \cdots + d_{\lfloor ({k_1 + \cdots + k_n})/{5}\rfloor}) \le \norm{A_n}\norm{Y_n} + \norm{X_n}\norm{B_n} \le (\norm{A_n}+\norm{B_n})(\norm{Y_n} + \norm{X_n})
  \end{equation}
  and for $N=3, i.e., D,C,Z$, we have
  \begin{equation}
    \label{eq4.7bgeneral}
    \frac{1}{k_n} (d_1 + d_2 + \cdots + d_{\lfloor ({k_1 + \cdots + k_n})/{7}\rfloor}) \le \norm{A_n}\norm{Y_n} + \norm{X_n}\norm{B_n} \le (\norm{A_n}+\norm{B_n})(\norm{Y_n} + \norm{X_n}).
  \end{equation}
  These left-hand ratio sequences (in these exponential cases are related to averages) decreases to $0$ as discussed in \Cref{rem1}.
  And so this lower bound provides quantitative constraints on the speed of tending to $0$ of the right-hand sides, which offers information on the norms of the finite blocks in the block tridiagonalizations of $C,Z$.
\end{remark}

\subsection*{Generalized Anderson's model (GAM): full finite blocks}

\Cref{equivthm}(c)--(d) indicates that allowing full block tridiagonal operators (as opposed to the more restricted forms of (EAM) in \eqref{blockseam}) is in a sense necessary and sufficient to solve the compact commutator problem with one exception.
Namely, that the target operator $D$, after the required basis change, is not in block tridiagonal form (though because $C,Z$ are simultaneously in block tridiagonal form with same compatible block sizes, $D$ is in block 5-diagonal form).
In contrast, in the beginning in \cite{And-1977-JRAM} Anderson directly forced a matrix solution for his commutator problem for a simple fixed rank one projection ($p_{11}=1, p_{ij} = 0$ elsewhere) rather than effecting a change of basis.
We do not know whether this strategy always works, so we leave it as one of the essential open questions in the subject:

\begin{problem}
  Is the change of basis in \Cref{t11}(b) necessary when $T$ is already block tridiagonal relative to the basis $\seq{e_n}$?
  That is, are the conditions in \Cref{t11} equivalent to: Every compact operator in our block tridiagonal form  \eqref{F} is a commutator of compact operators of the same form relative to the same basis, all 3 with central block sizes  $6 (7^{n-2})$.
  This requires solving the algebraic system of equations expanded from display \eqref{eqD_N} to compute commutators $T = CZ-ZC$ of the matrices simultaneously of the form \eqref{F}.
  Let
  \begin{equation}
    \label{TDF4.8}
    T = \begin{pmatrix}
      D_1 &  E_1  &0   & \cdots        \\
      F_1 & D_2&E_2  &  \ddots            \\
      0 &F_2 &D_3 & \ddots             \\
      \vdots &\ddots &\ddots  &\ddots  \\
    \end{pmatrix}, \quad
    C = \begin{pmatrix}
      C_1 &  A_1  &0   & \cdots        \\
      B_1 & C_2&A_2  &  \ddots            \\
      0 &B_2 &C_3 & \ddots             \\
      \vdots &\ddots &\ddots  &\ddots  \\
    \end{pmatrix} \quad and \quad
    Z = \begin{pmatrix}
      Z_1 & X _1  &0   & \cdots        \\
      Y_1 & Z_2&X_2  &  \ddots            \\
      0 &Y_2 &Z_3 & \ddots             \\
      \vdots &\ddots &\ddots  &\ddots  \\
    \end{pmatrix}
  \end{equation}
  with the central blocks $D_n, C_n, Z_n$ of size $k_n\times k_n$ ($k_n$ as in \Cref{t11} for $N=3$).
  Recall that the target operator $T$ is compact if and only if its three sequences of diagonal (main, upper and lower) finite blocks converge in norm to 0 (see for instance \Cref{thm:bandable} and recall that compact operators are those which take weakly convergent sequences of its Hilbert space vectors to norm convergent sequences).
  So Anderson's approach becomes:
  to solve the Pearcy--Topping problem it suffices to solve the system of matrix equations for $C,Z$ with their sequences of (main, upper and lower) diagonal blocks also converging in norm to 0.
  For the record, this expanded algebraic system is:
  \begin{equation}
    \label{eqProblem4.8}
    \begin{gathered}
      D_1=A_{1}Y_{1}-X_{1}B_{1} + \comm{C_1}{Z_1}, \quad \text{and for } n \ge 1,\\
      D_{n+1}=B_nX_n-Y_nA_n+A_{n+1}Y_{n+1}-X_{n+1}B_{n+1}+ \comm{C_{n+1}}{Z_{n+1}},\\
      A_nX_{n+1}=X_nA_{n+1}\quad\mbox{ and }\quad B_{n+1}Y_n=Y_{n+1}B_n, \\
      E_n=C_{n}X_{n}-Z_{n}A_{n} + A_nZ_{n+1}-X_nC_{n+1},\\
      F_n=B_nZ_n-Y_nC_n+C_{n+1}Y_n-Z_{n+1}B_n.
    \end{gathered}
  \end{equation}
  Observe that even in this generality the important relation in \eqref{eq:telescoping-partial-trace} and used often above is preserved:
  \begin{equation*}
    \trace(D_1+D_2+\dots+D_n)=\trace(A_nY_n - X_nB_n).
  \end{equation*}
  Also, when $T$ is simply a diagonal matrix $D$, importantly this system is much simpler because each $E_n, F_n$ is the zero matrix of the corresponding size, and all $D_n$ are square diagonal matrices.
  Thus this system has a better chance of being solvable or leading to counterexample obstructions.
\end{problem}

\subsection*{Further sparsification for block tridiagonal forms}

It is natural to ask whether \Cref{T1} can be further improved, i.e., can more zeros be obtained in the blocks.
Here, we are attempting to preserve the structure and block sizes as in  \Cref{T1}, while ensuring that some  additional entries in these blocks are universally zero.
\Cref{T2,t3aa}, whose proofs can be found in \cite{PPW-2020-TmloVLOt}, present some progress in this direction.

\begin{theorem}[\protect{\cite[Theorem 20.7]{PPW-2020-TmloVLOt}}]
  \label{T2}
  Every block tridiagonal matrix with diagonal block sizes $\seq{k_n}$ nondecreasing is unitarily equivalent to a block tridiagonal matrix with the same block sizes but also with $A_n$ of the form $(A'_n \mid 0)$ with $A'_n$ a positive square matrix.
  Alternatively, the same but with $B_n$ of the form $(B'_n \mid 0)^T$ with $B'_n$ a positive square matrix.
\end{theorem}

The next theorem \cite[Theorem 20.8]{PPW-2020-TmloVLOt}, is a further substantial sparsification of our canonical block tridiagonal \Cref{T1} case (i.e., with central block sizes $k_1=1, k_n = 2\cdot 3^{n-2}$).
However there is a subtle gap in the proof appearing in \cite{PPW-2020-TmloVLOt}, namely, the linearly dependent case at the end.
(See description just below.)
So here we give a different proof modeled after the proof of \cite[Theorem 2]{Wei-1980-IEOT}.
It is more complex than \cite[Theorem 20.8]{PPW-2020-TmloVLOt} so we first give a motivation.

Our study of block tridiagonal forms began with a study of staircase forms (see \Cref{t1} and preceding paragraph) via the sequence $e_1$, $Ce_1$, $C^*e_1$, $e_2$, $C^2e_1$, $C^*Ce_1$, $e_3$, $CC^*e_1$, ${C^*}^2e_1$, $e_4$, $Ce_2$, $C^*e_2$, $e_5$, $\dots$.
This sequence generates the staircase form \eqref{E0} when it is linearly independent, and it takes an induction proof to handle the retention of the form in the linearly dependent case.
The core idea of the next theorem is to redistribute the vectors $\seq{e_n}$ from arithmetically distributed (every third step) to geometrically distributed (every $3^{n-1}, n \ge 1$ step).
And again the difficult part, and where the error occurred, is in the linearly dependent case, which is naturally circumvented in the current proof.

We require the following simple lemma concerning the Gram--Schmidt process.

\begin{lemma}
  \label{lem:key-collapsing-lemma}
  Suppose that $\seq{g_n}$ is a linearly independent set such that there is an increasing sequence of positive integers $m_n$ and an operator $C$ such that
  \begin{equation*}
    C g_n \in \bigvee_{j=1}^{m_n} g_j \quad\text{for all } n \in \mathbb{N}.
  \end{equation*}
  Then if $\seq{f_n}$ is the orthonormal set obtained by applying the Gram--Schmidt process to $\seq{g_n}$, then
  \begin{equation*}
    C f_n \in \bigvee_{j=1}^{m_n} f_j \quad\text{for all } n \in \mathbb{N}.
  \end{equation*}
  Consequently, the matrix representation for $C$ relative to the orthonormal basis $\seq{f_n}$ has the property that, for the $n$-th column, the length of the support entries is at most $m_n$.
\end{lemma}

\begin{proof}
  Since $\seq{g_n}$ is a linearly independent set, the Gram--Schmidt process guarantees that $\bigvee_{j=1}^n f_j = \bigvee_{j=1}^n g_j$ for all $n$.
  Note that for all $n \in \mathbb{N}$, by hypothesis and the linearity of $C$,
  \begin{equation*}
    C \left( \bigvee_{j=1}^n g_j \right) = \bigvee_{j=1}^n Cg_j \subseteq \bigvee_{j=1}^n \left(\bigvee_{j=1}^{m_j} g_j \right) = \bigvee_{j=1}^{m_n} g_j,
  \end{equation*}
  where the latter equality holds since $m_n$ is an increasing sequence.
  Consequently,
  \begin{equation*}
    Cf_n \in C \left( \bigvee_{j=1}^n f_j \right) = C \left( \bigvee_{j=1}^n g_j \right) \subseteq \bigvee_{j=1}^{m_n} g_j = \bigvee_{j=1}^{m_n} f_j. \qedhere
  \end{equation*}
\end{proof}

\begin{theorem}
  \label{thm:arbitrary-staircase-forms}
  Let $\seq{r_k}_{k=0}^m$ be a countable (finite or infinite) collection of strictly increasing functions whose ranges together constitute a partition of $\mathbb{N}$, and let $r_0(1) = 1$.
  Let $\seq{T_k}_{k=1}^m \in \BH$ and let $\seq{e_n}$ be a spanning set for $\Hil$.

  Then there is a basis $\seq{f_n}$ derived from these data satisfying the following properties:
  \begin{enumerate}
  \item for all $n \in \mathbb{N}$, $e_n \in \bigvee_{j=1}^{r_0(n)} f_j$;
  \item for all $n \in \mathbb{N}$, and $1 \le k \le m$, $T_k f_n \in \bigvee_{j=1}^{r_k(n)} f_j$.
  \end{enumerate}
  Consequently, the sequence of column support lengths of $T_k$ in the basis $\seq{f_n}$ is given by $\seq{r_k(n)}_{n=1}^{\infty}$.

  Moreover, if $\seq{e_n}$ is an orthonormal basis and each $T_k$ is upper triangular with respect to this basis, then $e_n = f_n$ for all $n$.
\end{theorem}

\begin{proof}
  First note that for each $l \in \mathbb{N}$, $l \le r_0(l)$ since $r_0$ is strictly increasing.
  Moreover, for each $1 \le k \le m$, since $r_k$ is strictly increasing, $r_0(1) = 1$ and the ranges are disjoint, $l < r_k(l)$.

  We construct a linearly independent spanning set $\seq{g_n}$ inductively as follows.
  Set $g_1 := e_1 = e_{r^{-1}_0(1)}$.
  For $n > 1$, there is a unique $0 \le k \le m$ for which $n \in r_k(\mathbb{N})$.
  If $k = 0$, define $g_n$ to be the first vector in the list:
  \begin{equation}
    \label{eq:r0-gn-def}
    \{ e_{r_0^{-1}(n)}, e_{r_0^{-1}(n) + 1}, \ldots \} \setminus \bigvee_{j=1}^{n-1} g_j.
  \end{equation}
  Otherwise, if $1 \le k \le m$ define $g_n$ to be the first vector in the list:
  \begin{equation}
    \label{eq:rk-gn-def}
    \{ T_k g_{r_k^{-1}(n)}, e_2, e_3, \ldots \} \setminus \bigvee_{j=1}^{n-1} g_j.
  \end{equation}
  Note that $g_{r_k^{-1}(n)}$ is defined previously since $r_k^{-1}(n) < r_k(r_k^{-1}(n)) = n$.
  Moreover, in each case this list is nonempty since $\seq{e_n}$ is a spanning set for $\Hil$ and the span $\bigvee_{j=1}^{n-1} g_j$ is finite dimensional.
  Also, the sequence so constructed is linearly independent by design.

  Therefore, by construction, for each $n \in \mathbb{N}$, and for all $1 \le k \le m$
  \begin{equation*}
    e_n = e_{r_0^{-1}(r_0(n))} \in \bigvee_{j=1}^{r_0(n)} g_j
    \quad\text{and}\quad
    T_k g_n = T_k g_{r_k^{-1}(r_k(n))} \in \bigvee_{j=1}^{r_k(n)} g_j.
  \end{equation*}
  The first of these conditions guarantees that $\seq{g_n}$ spans $\Hil$ since $\seq{e_n}$ does, and hence applying Gram--Schmidt to $\seq{g_n}$ yields an orthonormal basis $\seq{f_n}$.
  By \Cref{lem:key-collapsing-lemma} and the fact that $\bigvee_{j=1}^n f_j = \bigvee_{j=1}^n g_j$ for all $n$, we may replace $\seq{g_n}$ with $\seq{f_n}$ everywhere in the previous display.

  Finally, if each $T_k$ is upper triangular relative to an orthonormal basis $\seq{e_n}$ one can show that $g_n = e_n$ for every $n$, and hence $f_n = e_n$ also.
  Indeed, it is a straightforward induction.
  Note that $g_1 = e_1$.
  Let $n > 1$ and suppose that $g_j = e_j$ for all $1 \le j < n$.
  When considering the definition of $g_n$, there are two cases depending on whether $n \in r_0(\mathbb{N})$ or $n \in r_k(\mathbb{N})$ for some $1 \le k \le m$.
  In the former case, $g_n$ is defined as the first vector in the list \eqref{eq:r0-gn-def}, which is clearly $e_n$ since $r_0^{-1}(n) \le n$ and using the inductive hypothesis.
  In the latter case, $g_n$ is defined as the first vector in the list \eqref{eq:rk-gn-def}.
  Note that by the inductive hypothesis and the fact that each $T_k$ is upper triangular with respect to the basis $\seq{e_l}_{l=1}^{\infty}$, and since $r_k^{-1}(n) < n$, 
  \begin{equation*}
    T_k g_{r_k^{-1}(n)} = T_k e_{r_k^{-1}(n)} \subseteq \bigvee_{j=1}^{r_k^{-1}(n)} e_j \subseteq \bigvee_{j=1}^{n-1} e_j = \bigvee_{j=1}^{n-1} g_j.
  \end{equation*}
  Hence, this first vector is removed from the list \eqref{eq:rk-gn-def}, so $g_n = e_n$.
\end{proof}

\begin{remark}
  \label{rem:countable-simultaneous-finite-support-lengths}
  Note that the previous theorem has a curious consequence.
  For any countably infinite collection of operators $\{T_k\}_{k=1}^{\infty}$, there is an orthonormal basis $\seq{f_n}$ relative to which the matrix representation of each $T_k$ has finite row and column support lengths (albeit not the same lengths for each operator).
  Indeed, simply partition $\mathbb{N}$ into an infinite collection $\{A_k\}_{k=0}^{\infty}$ of infinite sets with $1 \in A_0$.
  Then notice that each set in the partition induces a unique strictly increasing function $r_k : \mathbb{N} \to \mathbb{N}$ with $r_k(\mathbb{N}) = A_k$.
  Finally, apply \Cref{thm:arbitrary-staircase-forms} to the sequence $\seq{T'_k}_{k=1}^{\infty}$, where $T'_{2k-1} = T_k$ and $T'_{2k} = T_k^{*}$.
\end{remark}

\begin{corollary}
  \label{cor:single-operator-arbitrary-staircase}
  Let $r_0,r_1,r_2 : \mathbb{N} \to \mathbb{N}$ be strictly increasing functions whose ranges partition $\mathbb{N}$, and let $r_0(1) = 1$.
  Let $T \in \BH$ be arbitrary, and let $\seq{e_n}$ be any spanning set.
  Then there is an orthonormal basis $\seq{f_n}$ derived from these data relative to which sequences of column and row support lengths of $T$ are given by $\seq{r_1(n)}, \seq{r_2(n)}$, respectively.
  Moreover, $e_n \in \bigvee_{j=1}^{r_0(n)} f_j$ for all $n \in \mathbb{N}$.
\end{corollary}

\begin{proof}
  Apply \Cref{thm:arbitrary-staircase-forms} with $T_1 := T$ and $T_2 := T^{*}$.
\end{proof}

\begin{theorem}[\protect{\cite[Theorem 20.8]{PPW-2020-TmloVLOt}}]
  \label{t3aa}
  For arbitrary $C\in\BH$ and any orthonormal basis $\seq{e_n}$ of $\Hil$, there exists an orthonormal basis $\seq{f_n}$  in which $C$ has a block tridiagonal form as in \eqref{F} with the block sizes as in \Cref{T1} (with $k_1=1$) and:
  \begin{enumerate}
  \item each block $B_n=(B_n' \mid 0\mid 0)^T$ where all three blocks are square and $B_n'$ is upper triangular, i.e., $B_n'(i,j)=0$ if $i>j$;
  \item each block $A_n$ is of the form $(A'_n \mid A_n'' \mid 0)$ where all three blocks are square and
    $A_n''$ is lower triangular, i.e., $A_n''(i,j)=0$ if $i<j$.
  \item \label{item:e_n-in-3^n} $e_1=f_1$ and $e_n\in \vee_{k=1}^{3^n} f_k$ for all $n\in\mathbb{N}$.
  \end{enumerate}
  Alternatively, $C$ is unitarily equivalent to another matrix of the form \eqref{F} with the block sizes as in \Cref{T1} (with $k_1=1$), where each $A_n = (A'_n \mid 0\mid 0)$), all three blocks are square, and $A_n'$ is lower triangular, while each $B_n$ has the form $(B'_n \mid B_n''\mid 0)^T$, all three blocks are square, and $B''_n$ is an upper triangular matrix, and \ref{item:e_n-in-3^n} holds.
  (We alert the reader that herein our notation $k_n$ has been changed from $n_k$ in \cite{PPW-2020-TmloVLOt}.)
\end{theorem}

\begin{proof}
  We apply \Cref{thm:arbitrary-staircase-forms} with $T_1 := C$, $T_2 := C^{*}$, $r_0(n) := 3^n$ for $n > 1$, and $r_1(1) = 2$ and for $n > 1$,
  \begin{equation*}
    r_1(n) := 3^{\lceil \log_3 n \rceil} + n - 3^{\lceil \log_3 n \rceil - 1},
  \end{equation*}
  and $r_2 : \mathbb{N} \to \mathbb{N}$ is the unique strictly increasing function whose range is $\mathbb{N} \setminus (r_0(\mathbb{N}) \cup r_1(\mathbb{N}))$.
  In particular, the ranges of these functions are:
  \begin{align*}
    r_0(\mathbb{N}) &= \{ 1 \} \cup \{ 3^k \mid k > 1 \} \\
    r_1(\mathbb{N}) &= \{ 2 \} \cup \{ 3^k + j \mid k \ge 1, 1 \le j \le 2 \cdot 3^{k-1} \} \\
    r_2(\mathbb{N}) &= \{ 3 \} \cup \{ 3^k + j \mid k \ge 1, 2 \cdot 3^{k-1} < j < 2 \cdot 3^k \} \\
  \end{align*}
  \Cref{thm:arbitrary-staircase-forms} guarantees that the sequences of column and row (column of $C^{*}$) support lengths of $C$ are given by $\seq{r_1(n)}_{n=1}^{\infty}$ and $\seq{r_2(n)}_{n=1}^{\infty}$, respectively.
  Therefore, since $r_1(1) = 2$ and $r_2(1) = 3$, this yields $C_1,A_1,B_1$.
  Moreover, the column lengths being given by $r_1(n)$ precisely guarantees that each block $B_k$ (thinking of $k = \lceil \log_3 n \rceil + 1$) has the desired form with $B_k'$ upper triangular.

  Since the row support lengths for rows 2 and 3 are $r_2(2) = 6$ and $r_2(3) = 7$, this guarantees $A_2$ has the desired form.
  Each block of entries $\{ 3^k + j \mid 2 \cdot 3^{k-1} < j < 2 \cdot 3^k \}$ in $r_2(\mathbb{N})$ has size $4 \cdot 3^{k-1} - 1$.
  Therefore, if $r_2(n) = 3^{k'} + j$ for some $k' > 1$ and $2 \cdot 3^{k'-1} < j < 2 \cdot 3^{k'-1}$, then because $r_2$ is strictly increasing, we must have
  \begin{equation*}
    n = 1 + \sum_{i=1}^{k'-1} (4 \cdot 3^{i-1} - 1) + (j - 2 \cdot 3^{k'-1}) = j - k'.
  \end{equation*}
  Consequently, if also $3^{k-1} < n \le 3^k$ (for $k > 1$), then we have the inequality
  \begin{equation*}
    2 \cdot 3^{k'-1} - k' < j - k' = n \le 3^k,
  \end{equation*}
  which necessitates $k' \le k$.
  Therefore,
  \begin{equation*}
    r_2(n) = 3^{k'} + j = 3^{k'} + n + k' \le 3^k + n + 3^{k-1} = 3^k + (2 \cdot 3^{k-1}) + (n - 3^{k-1}).
  \end{equation*}
  Since $k = \lceil \log_3 n \rceil$, and $r_2(n)$ denotes the row support length, we see that $A_{k+1}$ has the desired form.
\end{proof}

\begin{remark}
  In the case where, at every stage, the initial vectors in the list obtained in the proof of \Cref{thm:arbitrary-staircase-forms} are linearly independent from what came before, the construction leads to the following list of vectors: $\seq{T_{*} g_{r_{*}^{-1}(n)}}_{n=1}^{\infty}$, where $*$ denotes the unique $0 \le k \le m$ for which $n \in r_k(\mathbb{N})$, and for convenience we define $T_0 g_{r_0^{-1}(n)} := e_{r_0^{-1}(n)}$.
  In particular, in the application of \Cref{thm:arbitrary-staircase-forms} (or more precisely, \Cref{cor:single-operator-arbitrary-staircase}) in the proof of \Cref{t3aa} above, this list is:
  \begin{align*}
    &e_1,Ce_1,C^*e_1,\\
    &\phantom{e_1,}
      Cg_2,Cg_3,\\
    &\phantom{e_1,}C^*g_2,C^*g_3,C^*g_4,\\
    &e_2,Cg_4,Cg_5,Cg_6,Cg_7,Cg_8,Cg_9,\\
    &\phantom{e_1,} C^*g_5,C^*g_6,C^*g_7,C^*g_8,..., C^*g_{15},\\&
    e_3,Cg_{10},Cg_{11},\dots
  \end{align*}
  where $g_2 = Ce_1$, $g_3=C^*e_1$, $g_4=Cg_2=C^2 e_1$, $g_5=Cg_3=CC^*e_1$, $g_6=C^*g_2=C^*Ce_1, \dots$.

  This is the same list which was constructed in the proof of \cite[Theorem 20.8]{PPW-2020-TmloVLOt} for the linearly independent case.
  So the induction proof given above subsumes this linearly independent case.
\end{remark}

\begin{remark}
  Both \Cref{T2} and \Cref{t3aa} exhibit a lack of symmetry regarding the role of blocks $\{A_n\}$ and $\{B_n\}$.
  We note that it should be possible to obtain a more symmetric sparsification of $A_n,B_n$ by modifying the functions $r_0,r_1,r_2$ in \Cref{t3aa}.

  In particular, set $r_0(n) = 2^{n+1} - n - 2$, and define $r_1(1) = 2$ and for all $r_0(k-1) < n \le r_0(k)$ with $k > 1$ (one can show $k = \lfloor \log_2 (2n) \rfloor - 1$), we set
  \begin{align*}
    r_1(n) &:= r_0(k) + 2 (n - r_0(k-1)) - 1 \\
    r_2(n) &:= r_1(n) + 1.
  \end{align*}
  Note that $r_0(n)$ satisfies the recurrence relation
  \begin{equation*}
    r_0(n+1) - r_0(n) = 2^{n+1} - 1 = 2(2^n - 1) + 1 = 2(r_0(n) - r_0(n-1)) + 1.
  \end{equation*}
  Then it is not hard to check that these satisfy the hypotheses of \Cref{thm:arbitrary-staircase-forms}.

  Therefore, taking $T_1 := C$, and $T_2 := C^{*}$, then as in the proof of \Cref{t3aa}, in the basis $\seq{f_n}$ provided by \Cref{thm:arbitrary-staircase-forms}, the operator $C$ has column and row support lengths given by $\seq{r_1(n)}, \seq{r_2(n)}$, respectively.
  Notice that
  \begin{equation*}
    r_2(r_0(n)) = r_0(n) + 2 (r_0(n) - r_0(n-1)) = r_0(n+1).
  \end{equation*}
  Consequently, because $r_i(r_0(n)) \le r_0(n+1)$ for all $i = 1,2$ and $n \in \mathbb{N}$, $C$ is in block tridiagonal form with block sizes given by $k_1 = 1 = r_0(1)$ and $k_n = r_0(n) - r_0(n-1) = 2^n - 1$ for $n > 1$.
  In this case, the off-diagonal blocks, which have an aspect ratio barely exceeding $2 : 1$, are more symmetrical.
  Indeed, the off-diagonal blocks are ``upper (lower) triangular'' in the sense that the $(i,j)$ entry of an off-diagonal block is zero if $i > 2j+1$ (respectively, $j > 2i +1$),
  i.e., the row and column support lengths increase by $2$ instead of $1$.
\end{remark}

\begin{remark}
  Reiterating our \cite[Problem 20.11]{PPW-2020-TmloVLOt}:
  Given an operator $C$, is there an orthonormal basis  in which
  $C$ is of the form as in \Cref{T2}, but where additionally each  $A'_n$ is a diagonal matrix?
  In particular, we wish to target the $5\times 5$ case as we did in \cite{PPW-2020-TmloVLOt}.
\end{remark}


\section{Support density}
\label{secdens}

In \Cref{sec2} we presented several universal block tridiagonal forms.
For each of them we strove to find a basis in which the matrix of a given operator is sparse and had a particularly useful form.
Now we will try to quantify their sparseness by considering the density of their support entries (meaning entries that need not be zero).
This is a basis dependent measure that depends also on the index ordering of the basis having index order type $\mathbb{N}$, i.e., order isomorphic to $\mathbb{N}$.

Matrices can often carry data in a kind of organized way which we think of as special matrix forms.
By a \term{matrix form}, we mean a subset of entry positions $(i,j)$ where matrices in the specified form must be zero;
we refer to the positions in the complement of this set as \term{support entries}.
For example, diagonal matrices are a matrix form corresponding to the subset $\{ (i,j) \mid i \neq j \}$.
Note that any given matrix in a specified form may have zeros as some of its support entries.

Laurent and Toeplitz infinite matrices carry data that is spread out, whereas diagonal matrices or band matrices have their data relegated close to the diagonal, and these matrices are quite sparse.
Similarly, any rank one projection operator can have a full matrix representation; but in some basis that rank one projection has all its support data relegated to its $(1,1)$ entry.
This indicates the inherent basis dependent nature of these matrix forms.
Our staircase and block tridiagonal forms also capture this idea by concentrating the support data, and our improvements at the end of the \Cref{sec2} sparsify these forms further still.
In this section we attempt to measure sparsification as compared to upper triangular matrices, upper Hessenberg matrices, diagonal matrices, etc.

Given a finite $N\times N$ matrix form, we define the number $L_N$ as the number of support entries in it, and the support density by the average $D_N=L_N/N^2$.
For an infinite matrix form $A$ we will consider the upper left $N\times N$ corners, i.e., the compressions $A_N :=P_NAP_N$ with projections $P_N$ onto $\bigvee_{n=1}^N e_n$.
The support density, or simply called density, of $A$ will mean
\begin{equation}
  \label{density}
  D=\lim_{N\to\infty} D_N=
  \lim_{N\to\infty} \frac{L_N}{N^2},
\end{equation}
whenever that limit exists.
Here $L_N$ and $D_N$ are calculated for the matrix $A_N$.

For several special classes of operators we briefly describe their special matrix forms and their densities.
By the spectral theorem for normal compact operators, such an operator can be represented by a diagonal matrix, and the density of the diagonal matrix form is obviously zero.
Moving away from the class of normal operators to the class of triangularizable operators, in the basis which triangularizes them all have density $1/2$.
Operators with a cyclic vector can be represented as an upper Hessenberg matrix.
Indeed, for cyclic vector $e$ apply the Gram--Schmidt process to $\seq{e, Te, T^2e, \dots }$.
The upper Hessenberg matrix form is visually similar to an upper triangular matrix except that its support entries on the first subdiagonal.
(In contrast, only triangularizable operators have bases relative to which their matrix representation is triangular.)
Upper Hessenberg matrices have density $1/2$, so all operators with cyclic vectors have a basis relative which their matrix representation has density $1/2$.
In the special case when the operator is selfadjoint the upper Hessenberg form reduces to a tridiagonal matrix, and the density of the tridiagonal matrix form (corresponding to $\{ (i,j) \mid \abs{i-j} > 1\}$) is clearly zero.
Without a cyclic vector the construction used to obtain the last two matrix forms (Hessenberg and tridiagonal) leads to a direct sum of these forms, but our definition of density is not compatible with direct sums of infinite matrices.

With density defined above we can go back to universal forms of \Cref{sec2}.
We begin with the staircase form \eqref{E0}.

\begin{theorem}
  \label{sec5a}
  The staircase form \eqref{E0} has density $2/3$.
\end{theorem}

\begin{proof}
  Recall that $(i,j)$ is a support entry of the staircase form \eqref{E0} if and only if $i \le 3j$ and $j \le 3i$.
  Taking the $N \times N$ compression, the diagram in \Cref{fig:staircase-density} shows that the density $D_N = L_N/N^2$ clearly approaches $2/3$.
  Of course, the values of $L_N$ given in the diagram are not exact since they may not be integers, but they are at least asymptotically equivalent (they differ from the true values by at most $6N$), which is all that matters.
  \begin{figure}[h]
    \centering
    \includegraphics{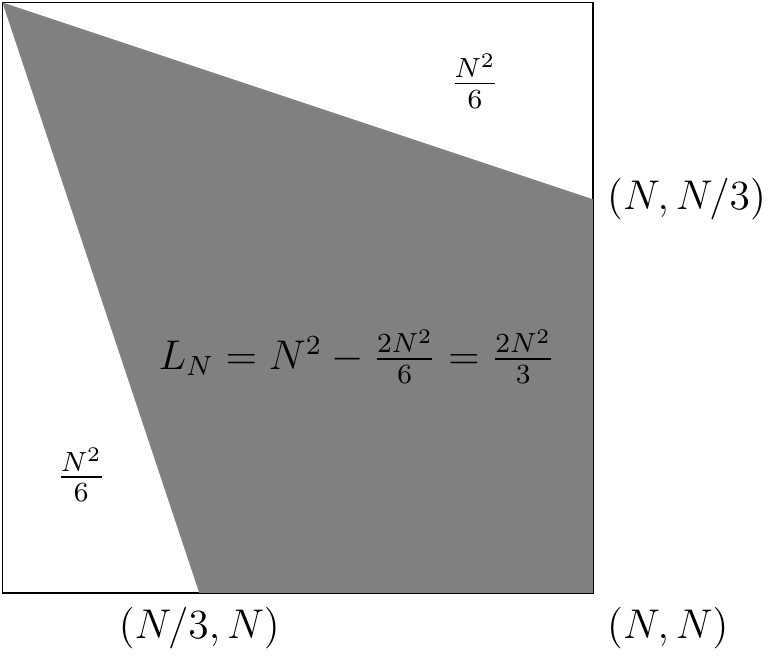}
    \caption{Asymptotic density of the staircase form \eqref{E0}.}
    \label{fig:staircase-density}
  \end{figure}
\end{proof}

It was noticed in \cite[Remark 20.3(i)]{PPW-2020-TmloVLOt} that, when $T$ has a cyclic vector (or more generally a vector $v$ for which the collection $p(T,T^*)v$ is dense), we obtain the matrix form where $(i,j)$ is a support entry if and only if $i \le 2j$ and $j \le 2i+1$.
For such matrices density is naturally smaller.
Since the proof follows the same template we leave it to the reader and merely state the result.

\begin{theorem}
  Every operator with cyclic vector has a basis relative to which its matrix is in staircase form with support rows and columns respectively $2,4,6,8,\dots$ and $3,5,7,\dots$, and these matrices have density $1/2$.
\end{theorem}

In \cite[Theorem 20.8]{PPW-2020-TmloVLOt} (\Cref{t3aa} herein) we gave a modification of the construction of \Cref{t1}.
It turns out that the sparsified matrices obtained that way have also density $1/2$ even when the operators from which they came do not have a cyclic vector.

\begin{theorem}
  \label{thm:arbitrary-staircase-density}
  The sparsified staircase form obtained in the proof of \Cref{t3aa} has density $1/2$.
  More generally, the staircase form of \Cref{cor:single-operator-arbitrary-staircase} has density $1/2$ whenever $\lim_{n \to \infty} n/r_0(n) = 0$.
\end{theorem}

\begin{proof}
  Consider the staircase form of \Cref{cor:single-operator-arbitrary-staircase} and let $L_N$ denote the number of support entries in the upper left $N \times N$ corner with the convention that $L_0 := 0$.
  Then for $N \ge 0$ set $\Delta_N := L_{N+1} - L_N$ to be the number of support entries appearing in column $N+1$ and row $N+1$ (within the $(N+1) \times (N+1)$ block).

  Of course, $\Delta_0 = 1$, and because $T$ is a \emph{staircase} form $\Delta_N \le \Delta_{N-1} + 2$ for all $N \ge 1$.
  This is because support entries in subsequent rows/columns of a staircase form can only occur below and/or to the right of those in the previous row/column.

  In fact, since the staircase form arises from \Cref{cor:single-operator-arbitrary-staircase}, $\Delta_N = \Delta_{N-1} + 2$ if and only if $N \in r_0(\mathbb{N})$, and otherwise $\Delta_N = \Delta_{N-1} + 1$.
  Indeed, if $N = r_1(m)$, then the $(N+1,m)$ position is not a support entry, and if $N = r_2(m)$, then the $(m,N+1)$ position is not a support entry, whereas if $N = r_0(m)$, then both $(N+1,m)$ and $(m,N+1)$ are support entries.
  Therefore, the sequence $\seq{\Delta_N}$ is the subsequence of positive integers which skips an integer at the indices $r_0(N)$ for each $N$.

  Consequently, let $N \in \mathbb{N}$ and suppose that $k$ is the unique positive integer such that $r_0(k) \le N < r_0(k+1)$.
  By filling in the missing integers (namely, $r_0(j) + j$ for $1 \le j \le k$) in the sequence $\seq{\Delta_N}$ and taking the sum, we obtain
  \begin{equation*}
    L_N + \sum_{j=1}^k (r_0(j) + j) = \sum_{i=0}^{N-1} \Delta_i + \sum_{j=1}^k (r_0(j) + j) = \sum_{j=1}^{N+k} j = (N+k)(N+k+1)/2.
  \end{equation*}
  Therefore,
  \begin{equation*}
    L_N = \frac{N^2 + 2kN + N}{2} - \sum_{j=1}^k r_0(j).
  \end{equation*}
  Since $0 < r_0(j) \le N$ for all $1 \le j \le k$, we find
  \begin{equation*}
    \frac{N^2 + kN + N}{2} \le L_N < \frac{N^2 + 2kN + N}{2},
  \end{equation*}
  and dividing through by $N^2$ yields
  \begin{equation*}
    \frac{1}{2} + \frac{k + 1}{2N} \le D_N < \frac{1}{2} + \frac{2k + 1}{2N}.
  \end{equation*}
  Finally, by hypothesis, as $N \to \infty$ so also $k \to \infty$, and $k/N \le k/r_0(k) \to 0$.
  Hence $D_N \to 1/2$.
\end{proof}

\begin{remark}
  The previous theorem applies to the staircase form obtained in the proof of \Cref{t3aa}, but not to the block tridiagonal form specified in the statement of the theorem, even taking into account the sparsification obtained for the off-diagonal blocks.
  In fact, the density of the block tridiagonal form given in \Cref{t3aa} can be shown to be $11/18$.
  The discrepancy is due to the fact that block tridiagonal form does not take full advantage of the row support lengths dictated by the corresponding staircase form, as can be seen by examining the proof.
\end{remark}

One could pose more general definitions of density for which the work above gives little information.
Moreover, the density is highly basis dependent.
We now show that even allowing permutations of certain bases, including those corresponding to our universal block tridiagonal form, can result in radically different densities, including density zero.

\begin{theorem}
  \label{thm:subsequence-density-permutation}
  Suppose that $T$ is a matrix form and $\seq{e_n}$ is an orthonormal basis of $\Hil$.
  Suppose there is a subsequence $\seq{e_{m_k}}_{k=1}^{\infty}$ such that the compression of $T$ to $\bigvee_{j=1}^k e_{m_j}$ has density $\delta$.
  Then there is a permutation $\pi$ of $\mathbb{N}$ for which $T$ has density $\delta$ relative to the orthonormal basis $\seq{e_{\pi(n)}}$.
\end{theorem}

\begin{proof}
  Let $\seq{e_{l_k}}$ be the complement subsequence of $\seq{e_{m_k}}$.
  Let $P_N$ denote the projection onto $\bigvee_{j=1}^N e_{m_j}$.
  Let $L_N$ denote the number of support entries of the compression of $T$ to $P_N \Hil$.
  Of course, by hypothesis $L_N/N^2 \to \delta$ as $N \to \infty$.

  Now define a permutation\footnotemark{} $\pi$ of $\mathbb{N}$ by
  \footnotetext{%
    For the proof, it suffices that the permutation $\pi$ satisfies $k/\pi^{-1}(l_k) \to 0$ and both $\pi^{-1}(l_k)$ and $\pi^{-1}(m_k)$ are strictly increasing.
  }
  \begin{equation*}
    \pi(n) :=
    \begin{cases}
      l_k & n = 2^{k-1}, \\
      m_{n - k} & 2^{k-1} < n < 2^k. \\
    \end{cases}
  \end{equation*}
  Let $Q_N$ denote the projection onto $\bigvee_{j=1}^N e_{\pi(j)}$.
  Notice that if $2^{k-1} \le N < 2^k$ (i.e., $k = \lfloor\log_2(N)\rfloor + 1$), then
  \begin{equation*}
    \{e_{\pi(1)},\ldots,e_{\pi(N)}\} = \{e_{l_1},\ldots,e_{l_k}\} \cup \{e_{m_1},\ldots,e_{m_{N - k}}\}.
  \end{equation*}
  Then let $L'_N$ denote the number of support entries of the compression of $T$ to $Q_N$, then
  \begin{equation*}
    L_{N-k} \le L'_N \le L_{N-k} + 2kN.
  \end{equation*}
  Notice that $\lim_{N \to \infty} \frac{N-k}{N} = 1$, and hence $\lim_{N \to \infty} \frac{L_{N-k}}{N^2} = \delta$.
  Consequently, by this and the above display, $L'_N/N^2 \to \delta$ as $N \to \infty$, and so $T$ has density $\delta$ in the basis $\seq{e_{\pi(n)}}$.
\end{proof}

\begin{corollary}
  \label{cor:row-column-finite-density-zero}
  Suppose that $T$ is a matrix form and $\seq{e_n}$ is an orthonormal basis of $\Hil$ in which the support lengths of all the rows and columns of $T$ are finite.
  Then there is a permutation $\pi$ of $\mathbb{N}$ for which $T$ has density $0$ relative to the basis $\seq{e_{\pi(n)}}$.
\end{corollary}

\begin{proof}
  Since all the rows and columns of $T$ have finite support lengths, there is a subsequence $\seq{e_{m_k}}$ for which the compression of $T$ to $\bigvee_{k=1}^{\infty} e_{m_k}$ is diagonal, and therefore has density $0$ relative to this basis for the subspace.
  By \Cref{thm:subsequence-density-permutation}, there is a permutation $\pi$ of $\mathbb{N}$ for which $T$ has density $0$ relative the basis $\seq{e_{\pi(n)}}$.
\end{proof}

\begin{corollary}
  \label{cor:all-density-zero}
  For every operator $T\in\BH$, there is some orthonormal basis $\seq{e_n}$ in which $T$ has density $0$.
\end{corollary}

\begin{proof}
  By \Cref{cor:row-column-finite-density-zero}, we may take the basis to be a permutation of the basis arising from any staircase form for $T$.
\end{proof}

\section{Constraints on the Pearcy--Topping problem from $s$-numbers and $\BH$-ideals}
\label{sec6}

Much of this paper is focused on some internal matrix structure constraints on the search for compacts $C,Z$ when solving $T = CZ-ZC$.
In this section, we provide some elementary $s$-number quantitative constraints (that more generally  apply to finite sums of bi-products of operators) and some consequent $\BH$-ideal algebraic qualitative constraints with some examples and some problems.
For these we need the characterization: compact operators are those whose $s$-number sequence tends to 0.
Unlike the past sections concentrating on the innards of certain matrix representations, some of the conditions presented in this section are independent of bases representations of the operators $C,Z$.
We also provide an $s$-number and ideal analysis of Anderson's original contruction for the rank one projection $P = CZ-ZC$ along with an analysis of the relationship between his $C$ (where $C = C_1 + C_{-1}$) and its two block upper and lower diagonals $C_1$, $C_{-1}$ and their principal ideals $\ideal{C}$, $\ideal{C_{1}}$, and $\ideal{C_{-1}}$, and similarly for $Z$.
This analysis originated in \cite{LW-2021-OM}.
A direct $s$-number approach provides necessary quantitative $s$-number conditions for $T = CZ-ZC$ (see \Cref{$s$-number proposition}).

We will not assume that the reader is familiar with the theory of $\BH$-ideals or $s$-numbers.
The $s$-numbers of a compact operator $T$ are defined as $s_n(T) := \inf\{\norm{T-F} \mid \rank F < n\}$ for $n \in \mathbb{N}$, and it is well-known that the sequence $s(T) := \seq{s_n(T)}_{n=1}^{\infty}$ lists the eigenvalues of $\abs{T}$ with multiplicity in nonincreasing order (omitting the eigenvalue $0$ if $T$ has infinite rank).
Unless otherwise indicated, all operations on sequences described in this section are to be performed pointwise, e.g., $\sqrt{s(T)} := \seq{\sqrt{s_n(T)}}$.
Given sequences $\lambda,\mu$, we let $\lambda \vee \mu, \lambda \wedge \mu$ denote their pointwise maximum and minimum, respectively.

For this section we will recall the Calkin--Schatten characterization for $\BH$-ideals in terms of characteristic sets (see for instance \cite[Section 2.7 end]{DFWW-2004-AM}).
To any proper $\BH$-ideal $\J$ one can associate its \emph{characteristic set}
\begin{equation*}
  \charset(\J) := \{ s(T) \mid T \in \J \}.
\end{equation*}
Then $\charset(\J)$ is a subcone of the cone of nonnegative nonincreasing sequences converging to zero, and moreover this subcone is hereditary (under the pointwise order) and closed under the $k$-ampliation\footnotemark{} operator $D_k(\seq{a_n}) := \seq{a_{\lceil n/k \rceil}}$.
Likewise, given any such subcone $\charset$, one can define the associated ideal
\begin{equation*}
  \J_{\charset} := \{ T \in \KH \mid s(T) \in \charset \}.
\end{equation*}
The map $\J \mapsto \charset(\J)$ is an order (and semiring) isomorphism from the (subset) lattice of proper operator ideals into the (subset) lattice of hereditary subcones closed under $2$-ampliation, and its inverse is $\charset \mapsto \J_{\charset}$.
For details consult \cite[p. 26, Theorem 12]{Sch-1970}

\footnotetext{%
  Since $\charset(\J)$ is hereditary, it is closed under $2$-ampliation if and only if it is closed under $k$-ampliation for any $k \ge 2$.%
}

The Calkin--Schatten characterization often significantly simplifies analysis in the theory of operator ideals and allows for ideal constructions by defining the corresponding construction on characteristics sets;
this is a key feature often exploited in \cite{DFWW-2004-AM}.
As an example, any operator ideal $\J$ has a natural well-defined square root $\J^{1/2}$ whose characteristic set is given by $\charset(\J^{1/2}) := \{ \sqrt{s(T)} \mid s(T) \in \charset(\J) \}$, which is a hereditary subcone closed under ampliations.

Recall two basic $s$-number inequalities \cite[Corollary 2.2]{GK-1969-ITTTOLNO}:
\begin{equation*}
  s_{n+m-1}(AB) \le s_n(A)s_m(B) \quad \text{and}\quad s_{n+m-1}(A+B) \le s_n(A) + s_m(B),
\end{equation*}
which by straightforward induction arguments establish:
\begin{equation*}
  s_{m_1 + \cdots + m_k - (k-1)}\left(\prod^k_{i=1} A_i\right)  \le  \prod^k_{i=1}s_{m_i}(A_i) \quad \text{and} \quad  s_{m_1 + \cdots + m_k - (k-1)}\left(\sum^k_{i=1} A_i\right)  \le  \sum^k_{i=1}s_{m_i}(A_i).
\end{equation*}
For $m \in \mathbb{N}$, we may divide $m$ by $k$ to get a quotient $q = \lceil m/k \rceil$ with nonpositive remainder $-r$ with $0 \le r < k$, so that $m = qk - r \ge qk - (k-1)$.
Then from the above display, we obtain:
\begin{equation*}
  s_m(A_1 \cdots A_k) = s_{qk - r}(A_1 \cdots A_k) \le s_{qk - (k-1)}(A_1 \cdots A_k) \le s_q(A_1) \cdots s_q(A_k),
\end{equation*}
and similarly for $s_m(A_1 + \cdots + A_k)$.
Therefore since $q = \lceil m/k \rceil$ and $D_k(\seq{a_m}) = \seq{a_{\lceil m/k \rceil}}$,
\begin{equation}
  \label{eq:ampliation-product-sum-s-numbers}
  s(A_1 \cdots A_k) \le D_k(s(A_1) \cdots s(A_k)) \quad\text{and}\quad s(A_1 + \cdots + A_k) \le D_k(s(A_1) + \cdots + s(A_k)).
\end{equation}

\subsection*{A quantitative $s$-number sequence analysis for $T = CZ -ZC$.}

\hfill

\noindent
Clearly, $T \in \ideal{CZ} + \ideal{ZC} \subseteq \ideal{C} \ideal{Z}$ is a qualitative constraint, but an explicit quantitative constraint is:

\begin{proposition}
  \label{$s$-number proposition}
  If $T = \sum_{i=1}^k C_i Z_i$, with each $C_i, Z_i$ compact, then
  \begin{equation*}
    s(T) \le \sum_{i=1}^k D_{2k} (s(C_i) s(Z_i)).
  \end{equation*}

  In particular, if $T = CZ-ZC$ for $C,Z$ both compact, then $s(T) \le 2 D_4(s(C) s(Z))$.
\end{proposition}

\begin{proof}
  By \eqref{eq:ampliation-product-sum-s-numbers}, and noting that the ampliation operations are semiring endomorphisms and $D_{2k} = D_k \circ D_2$,
  \begin{equation*}
    s(T) = s \left( \sum_{i=1}^k C_i Z_i \right) \le D_k \left( \sum_{i=1}^k s (C_i Z_i) \right) \le D_k \left( \sum_{i=1}^k D_2 (s(C_i) s(Z_i)) \right) = D_{2k} \left( \sum_{i=1}^k s(C_i) s(Z_i) \right).
  \end{equation*}
  For $T = CZ - ZC$, since $s(-Z) = s(Z)$, the above yields $s(T) \le 2 D_4(s(C) s(Z))$.
\end{proof}

\begin{remark}
  \label{sums of bi-products}
  Since \Cref{$s$-number proposition} applies equally well to $T = CZ + ZC$, and even more generally to finite sums of bi-products, this proposition does not distinguish commutators on its own.
\end{remark}

In any case, recalling historical facts on sums of commutators:
Every bounded operator is the sum of two commutators of bounded operators (Brown--Pearcy and Brown--Halmos--Pearcy, see \cite[p.\,248, paragraph below Problem 1, and citations]{PT-1971-MMJ});
Every compact operator is the finite sum of commutators of compact operators, i.e., $\KH = \comm{\KH}{\KH}$ \cite[Theorem 1]{PT-1971-MMJ};
and from this in \cite[Corollary~6.5]{DFWW-2004-AM} it follows that every compact operator is the sum of four commutators, i.e., $\KH = \comm{\KH}{\KH}_4$, although \cite[Corollary~6.5]{DFWW-2004-AM} applies more generally to the commutator space $\comm{\I}{\J} = \comm{\I}{\J}_4$ for arbitrary operator ideals $\I,\J$.
For $\KH$ specifically, in fact even more is known: every compact operator is the sum of two commutators of compact operators, which is due to Fan and Fong in \cite{FF-1987-PRIA}.

So for $T \in \KH$, one has $T = \sum^2_{i = 1}\comm{C_i}{Z_i}$.
For these operators applying \Cref{$s$-number proposition} one obtains
\begin{equation}
  \label{eq:D8-constraint}
  s(T) \le 2 D_8 (s(C_1) s(Z_1) + s(C_2) s(Z_2)).
\end{equation}
While we don't know even for all strictly positive operators $T$ that a representation as a single commutator exists (although we provided a class of such strictly positive compact operators in \Cref{sec3}), we note that if $T = CZ - ZC$, then \Cref{$s$-number proposition} affords the alternate constraint
\begin{equation}
  \label{eq:D4-constraint}
  s(T)  \le 2 D_4 (s(C)s(Z)).
\end{equation}

\subsection*{A qualitative analysis of $\BH$-ideal constraints and open questions}

The previous quantitative sharpening of the constraints has no bearing on the following ideal analysis, but the qualitative implications of the $s$-number inequalities do.

A purely algebraic approach tells us: $T = CZ-ZC \Rightarrow T \in \ideal{C}\ideal{Z}$;
indeed, ideal product is a commutative operation, and also $\ideal{T} \subseteq \ideal{C}\ideal{Z} = \ideal{\diag (s(C)s(Z))}$ (see \cite[Section 2.8 and Theorem 6.3]{DFWW-2004-AM}).

Thus we have the following interesting question for compact $C,Z$:

\begin{question}
  When is it possible to choose $C,Z$ such that $T = CZ-ZC$ and $\ideal{T} = \ideal{C}\ideal{Z}$?
\end{question}

For the remainder of this subsection, we analyze several different natural scenarios.
Note that if $T = \comm{C}{Z}$, then $\ideal{T} \subseteq \ideal{CZ} + \ideal{ZC} \subseteq \ideal{C} \ideal{Z} \subseteq \ideal{C} \cap \ideal{Z}$.
In general, any of these ideal inclusions may be proper.
We note that there are known instances in which equality (even $\ideal{T} = \ideal{C} \cap \ideal{Z}$) can be achieved, but also established instances for which the containment is forced to be proper (even $\ideal{T} \subsetneq \ideal{CZ} + \ideal{ZC}$).

\begin{example}
  \label{ex:containment-equality}
  Given any sequence $\seq{d_n}$, the operator $T = \diag{\seq{d_1, -d_1, d_2, -d_2, \ldots}}$ can be written as the commutator
  \begin{equation*}
    T = \comm{C}{Z} := \comm*{\bigoplus \sqrt{d_n}\left(\begin{matrix} 0 & 1\\ 0 & 0 \end{matrix}\right)}{\bigoplus \sqrt{d_n}\left(\begin{matrix} 0 & 0\\ 1 & 0 \end{matrix}\right)}
  \end{equation*}
  (This is an example of broader phenomena, selfcommutators $\comm{A}{A^*} = A^*A-AA^*$ of a compact operator $A$ which have been characterized by P.~Fan and C.~Fong \cite[Theorem 1]{FF-1980-PAMS}.)
  Of course, when $\seq{d_n}$ is nonnegative nonincreasing and converging to zero, then $T$ is compact, and moreover, one can check that $s(T) = D_2 (\seq{d_n})$ and $s(C) = s(Z) = \seq{\sqrt{d_n}}$.
  Consequently, by appeal to characteristic sets, we easily obtain $(T) = (C)(Z)$.

  In fact, certain choices of $\seq{d_n}$, for example $d_n := 2^{-n}$, even yield $(T) = (C) = (Z)$.
  Indeed, in this case, notice that $\sqrt{d_n} = 2^{-n/2} \le 2^{-\lfloor n/2 \rfloor} = 2^{1/2} \cdot 2^{-\lceil n/2 \rceil}$, and hence $s(C) = s(Z) \le \sqrt{2} \seq{2^{-\lceil n/2 \rceil}} = \sqrt{2} s(T)$.
  Thus $C,Z \in \ideal{T}$, hence equality of the corresponding ideals.
  Of course, in general, $(C)(Z) \subseteq (C) \cap (Z)$ is a proper inclusion (e.g., $d_n = 1/n$, because $s(C) \neq O(D_k(s(T)))$ for any $k$), so we would rarely expect these three principal ideals to be equal, or even merely $\ideal{C} \ideal{Z} = \ideal{C} \cap \ideal{Z}$.
\end{example}

In the previous example, whenever the operator $T$ is trace-class (i.e., $\seq{d_n}$ absolutely summable), its trace is zero.
This may lead one to wonder about trace-class $T$ with nonzero trace.
We now concern ourselves primarily with this situation.

\begin{example}
  \label{ex:trace-class-nonzero-trace}
  Suppose that $T = \comm{C}{Z}$ is trace-class with $C,Z \in \BH$.
  If either of $CZ, ZC$ are trace-class, then both would be (since $T$ is), and hence \cite[Lemma~2.1]{LNRR-1981-JRAM} guarantees $\trace(CZ) = \trace(ZC)$, so $\trace(T) = \trace(\comm{C}{Z}) = 0$.
  Therefore, if $\trace(T) \neq 0$, then we must have $CZ, ZC \notin \traceclass$.
  Consequently, $T$ trace-class with $\trace(T) \neq 0$ implies $\ideal{T} \subsetneq \ideal{CZ} + \ideal{ZC} \subseteq \ideal{C}\ideal{Z}$.
\end{example}

\begin{remark}
  \label{rem:intersection-trace-class}
  A sufficient condition for $CZ,ZC \in \schatten{1}$ (and hence $\trace(T) = \trace(\comm{C}{Z}) = 0$) is $\ideal{C} \cap \ideal{Z} \subseteq \schatten{1}$.
  One may wonder whether the inclusion $\ideal{C} \cap \ideal{Z} \subseteq \schatten{1}$ is even possible unless either $C$ or $Z$ is itself trace-class.
  We remark that this is indeed possible even if $C,Z \notin \schatten{1}$.

  First, it is elementary in the theory of operator ideals using characteristic sets to prove that $\ideal{C} \cap \ideal{Z} = \ideal{\diag (s(C) \wedge s(Z))}$.
  So, it suffices to exhibit nonsummable sequences $c,z \in \czstar$ ($=s(C),s(Z)$) for which $c \wedge z$ is summable.

  We construct $c,z$ as follows.
  For $k \in \mathbb{N}$, set $n_k := k(k+1)/2$, and set $s_k := \sum_{j=1}^k 2^{n_j}$ and for convenience set $s_0 := 0$.
  Then define,
  \begin{gather*}
    c_m := 2^{-n_{2k} + 1}
    \quad\text{for}\quad
    s_{2(k-1)} < m \le s_{2k}, k \in \mathbb{N} \\
    z_m :=
    \begin{cases}
      1 & 1 \le m \le s_1 \\
      2^{-n_{2k+1} + 1}  & s_{2k-1} < m \le s_{2k+1}, k \in \mathbb{N} \\
    \end{cases}
  \end{gather*}
  Notice that, by construction,
  \begin{equation*}
    (c \wedge z)_m = 2^{-n_{k+1} + 1}
    \quad\text{for}\quad
    s_{k-1} < m \le s_k, k \in \mathbb{N}.
  \end{equation*}
  Therefore,
  \begin{equation*}
    \sum_{m=1}^{\infty} (c \wedge z)_m = \sum_{k=1}^{\infty} \sum_{m=s_{k-1} + 1}^{s_k} (c \wedge z)_m = \sum_{k=1}^{\infty} \sum_{m=s_{k-1} + 1}^{s_k} 2^{-n_{k+1} + 1} = \sum_{k=1}^{\infty} 2^{n_k} 2^{-n_{k+1} + 1} = \sum_{k=1}^{\infty} 2^{-k} = 1.
  \end{equation*}
  On the other hand,
  \begin{equation*}
    \sum_{m=1}^{\infty} c_m = \sum_{k=1}^{\infty} \sum_{m=s_{2(k-1)} + 1}^{s_{2k}} 2^{-n_{2k} + 1} = \sum_{k=1}^{\infty} (2^{n_{2k-1}} + 2^{n_{2k}}) 2^{-n_{2k} + 1} \ge \sum_{k=1}^{\infty} 2 = \infty,
  \end{equation*}
  and
  \begin{equation*}
    \sum_{m=1}^{\infty} z_m \ge \sum_{k=1}^{\infty} \sum_{m=s_{2k-1} + 1}^{s_{2k+1}} 2^{-n_{2k+1} + 1} = \sum_{k=1}^{\infty} (2^{n_{2k}} + 2^{n_{2k+1}}) 2^{-n_{2k+1} + 1} \ge \sum_{k=1}^{\infty} 2 = \infty.
  \end{equation*}
\end{remark}

\begin{example}
  \label{ex:brown-trace-zero}
  The earlier \Cref{ex:trace-class-nonzero-trace} implies that if $C \in \schatten{p}, Z \in \schatten{q}$ (Schatten ideals) and $T = \comm{C}{Z}$ has nonzero trace, then $1/p + 1/q < 1$ (or equivalently, if $1/p + 1/q \ge 1$, then $T = \comm{C}{Z}$ has zero trace).
  Rather interestingly, Brown proved \cite{Bro-1994-JRAM} that when the commutator $T = \comm{C}{Z}$ is finite rank, then even more can be said, namely, if $1/p + 1/q \ge 1/2$ (hence $CZ,ZC \in \schatten{2}$), then $\trace(T) = 0$.

  In case $P$ is a rank-one projection and $P = \comm{C}{Z}$ then at least one of $C,Z \notin \schatten{p}$ for all $p \le 4$.
  In Anderson's construction with $t = 1/2$, \cite[Theorem~3]{And-1977-JRAM} establishes $C,Z \in \bigcap_{p > 4} \schatten{p}$, so in some sense it appears that Anderson's construction may be optimal.
  In \Cref{lem:anderson-ideal} below, we compute explicitly the principal ideal $\ideal{C} = \ideal{Z}$ arising from Anderson's construction with $t = 1/2$.
  It would be interesting to know if Anderson's construction really is minimal (see \Cref{que:anderson-minimal}).
\end{example}

The above examples lead us to pose the following questions:

\begin{question}
  \label{que:positive-always-proper}
  If $T = \comm{C}{Z}$ is positive and not trace-class, is the inclusion $\ideal{T} \subseteq \ideal{C}\ideal{Z}$ always proper?
\end{question}

\begin{question}
  \label{que:hilbert-schmidt-commutator-trace-zero}
  If $T = \comm{C}{Z}$ is trace-class and $CZ,ZC \in \schatten{2}$, do these conditions imply $\trace(T) = 0$?
\end{question}

\Cref{que:positive-always-proper} is motivated by \Cref{ex:containment-equality,ex:trace-class-nonzero-trace}.
In particular, \Cref{ex:containment-equality} shows that in order to guarantee the inclusion is proper in the case when $T = \comm{C}{Z}$ is not trace-class, one needs some additional hypotheses.
Given that all the operators $T$ from \Cref{ex:trace-class-nonzero-trace} in some sense have ``trace zero'' (because the positive and negative parts are unitarily equivalent), \Cref{ex:trace-class-nonzero-trace} suggests the natural test case is when this is asymmetric, i.e., $T$ is positive and not trace-class.

\Cref{que:hilbert-schmidt-commutator-trace-zero} is motivated from \Cref{ex:brown-trace-zero} and \cite{Bro-1994-JRAM}.
In the proof in \cite{Bro-1994-JRAM}, the Schatten class assumption of the individual operators $C, Z$ was only used to bound Hilbert--Schmidt norms associated to submatrices of the products $CZ, ZC$, suggesting that maybe there is another approach which replaces this assumption only with $CZ, ZC \in \schatten{2}$.
Likewise, the finite rank assumption on $T$ was already partially loosened in \cite{Bro-1994-JRAM} to $T$ trace-class with $\comm{C}{C^{*}}$ finite rank.

Before we compute explicitly the principal ideals $\ideal{C}, \ideal{Z}$ arising from Anderson's construction $P = [C,Z]$ for a rank-one projection $P$, we first recall for the reader some basic facts about nonnegative nonincreasing sequences converging to zero which will be used below.
For readers intimately familiar with the theory of operator ideals, these facts will be well-known;
for others, they should be comprehensible with a bit of thought.

Suppose $\lambda = \seq{\lambda_n}, \mu = \seq{\mu_n}$ are positive sequences converging to zero and suppose $\lambda \le \mu$.
Then if $\lambda^{\star}, \mu^{\star}$ denote their nonincreasing rearrangements, $\lambda^{\star} \le \mu^{\star}$ (when a sequence has infinite support, any zeros in the sequence are omitted in its nonincreasing rearrangement).
Moreover, the collection $\czstar$ of nonnegative nonincreasing sequences converging to zero has an \emph{internal direct sum} (cf. \cite[\textsection 2.4]{DFWW-2004-AM}) operation given by $\lambda, \mu \in \czstar \mapsto (\lambda \oplus \mu)^{\star}$, where $(\lambda \oplus \mu)_{2n} := \lambda_n$ and $(\lambda \oplus \mu)_{2n+1} := \mu_n$.
We note that $\lambda \vee \mu \le (\lambda \oplus \mu)^{\star} \le D_2(\lambda \vee \mu)$.

In what follows, we will also be concerned with finite sequences.
Suppose that for each $n \in \mathbb{N}$, $a_n$ denotes some finite sequence.
Then $\bigoplus_n a_n$ is the concatenation of these sequences.
If the finite sequences $a_n$ are nonnegative and, for any $\epsilon > 0$, there are only finitely many $n$ such that any term in $a_n$ exceeds $\epsilon$, then $\bigoplus_n a_n$ is nonnegative and converges to zero.
In this case, we note that the monotonization and $k$-ampliation operators ($(\cdot)^{\star}$ and $D_k(\cdot)$) are well-behaved with respect to the direct sum.
More precisely,
\begin{equation*}
  D_k \left( \left( \bigoplus a_n \right)^{\star} \right) = \left( \bigoplus D_k(a_n) \right)^{\star}.
\end{equation*}
Of course, $D_k$ applied to a sequence of length $m$ produces a sequence of length $km$.

\begin{lemma}
  \label{lem:anderson-ideal}
  For Anderson's commutator construction for the rank-one projection $P = CZ-ZC$, 
  $\ideal{C} = \ideal{Z} = \ideal{\diag\seq{1/\sqrt[4]{n}}}$.
\end{lemma}

\begin{proof}
  In the notation of \Cref{thm:bandable} below, $C_{\pm 1}, Z_{\pm 1}$ denote the upper and lower block diagonal operators of $C,Z$, respectively.
  Moreover, by \Cref{thm:bandable}, we have $\ideal{C} = \ideal{C_1} + \ideal{C_{-1}}$ because $C_0 = 0$, and similarly for $Z$.
  A simple computation shows that $C_{\pm 1}^{*} C_{\pm 1}$ and $Z_{\pm 1}^{*} Z_{\pm 1}$ are diagonal, and hence these diagonal entries are precisely the singular values $s(C_{\pm 1}^{*} C_{\pm 1}) = s(\abs{C_{\pm 1}}^2) = s(C_{\pm 1})^2$, and similarly for $s(Z_{\pm 1})$.

  For each $n \in \mathbb{N}$, consider the finite sequence $c_n := \seq{(n-k+1)/n^2}_{k=1}^n$.
  Since none of the terms in $c_n$ exceed $1/n$, it is clear that $c := \bigoplus_n c_n$ is nonnegative and converges to zero.
  Examination of the definitions of $C_{\pm 1}, Z_{\pm 1}$ (see \cite[p. 7]{BPW-2014-VLOT} or the original \cite[pp. 129--130]{And-1977-JRAM} after setting his parameter $t = 1/2$) shows that the nonzero diagonal entries of each of $C^{*}_{\pm 1} C_{\pm 1}, Z^{*}_{\pm 1} Z_{\pm 1}$ are precisely given by the sequence $c$ (up to permutation).
  Therefore,
  \begin{equation*}
    s(Z_{\pm 1})^2 = s(C_{\pm 1})^2 = c^{\star},
    \quad\text{hence}\quad
    \ideal{Z_{\pm 1}}^2 = \ideal{C_{\pm 1}}^2 = \ideal{\diag c^{\star}}.
  \end{equation*}
  And thus
  \begin{equation*}
    \ideal{Z} = \ideal{C} = \ideal{C_{-1}} + \ideal{C_1} = \ideal{C_1} = \ideal1{\diag \sqrt{c^{\star}}}.
  \end{equation*}
  
  Now, note that $2 D_2(c_n)$ is a finite sequence of length $2n$, and we may let $d_n, d_n'$ denote the sequences of length $n$ so that $2 D_2(c_n)$ is the concatenation $d_n \oplus d'_n$ of $d_n$ followed by $d_n'$.
  We remark that $d_n \ge d_n'$ (since $2 D_2(c_n)$ is nonincreasing).
  We let $d := \bigoplus_n d_n$, $d' := \bigoplus_n d'_n$;
  clearly, $d \ge d'$ and so $d^{\star} \ge d'^{\star}$.
  Moreover, notice that $d_n \ge a_n := \seq{1/n}_{k=1}^n$.
  Therefore, $d \ge a := \bigoplus_n a_n$, and so $d^{\star} \ge a^{\star} = a$.

  Putting all this together yields
  \begin{align*}
    2 D_2(c^{\star})
    &= \left(\bigoplus 2 D_2(c_n) \right)^{\star}
      = \left( \bigoplus (d_n \oplus d_n') \right)^{\star} \\
    &= (d \oplus d')^{\star}
      = (d^{\star} \oplus d'^{\star})^{\star}
      \ge d^{\star} \vee d'^{\star} \\
    &\ge d^{\star}
      \ge a^{\star} = a.
  \end{align*}
  In addition, we remark that $c_n \le a_n$, and hence $c \le a$, so $c^{\star} \le a^{\star} = a$.
  These inequalities (namely, $c^{\star} \le a$ and $a \le 2D_2(c^{\star})$) guarantee that the characteristic sets generated by $c^{\star}$ and $a$ coincide, or equivalently, the principal ideals $\ideal{\diag c^{\star}}$ and $\ideal{\diag a}$ are equal.

  Consider now the sequence $\seq{1/\sqrt{n}} \in \czstar$.
  We show that $\seq{1/\sqrt{n}}$ and $a$ generate the same characteristic set
  Given $n \in \mathbb{N}$, there is some $k \in \mathbb{N}$ for which $k(k+1)/2 \le n < (k+1)(k+2)/2$, and hence the $n$-th term of $a$ lies in the finite constant sequence $a_k$, and therefore takes the value $1/k$.
  Dividing the inequalities through by $k^2$, we see that
  \begin{equation*}
    \frac{1}{2} \le \frac{1}{2} + \frac{1}{2k} \le \frac{n}{k^2} < \left(\frac{1}{2} + \frac{1}{2k}\right)\left(1 + \frac{2}{k}\right) \le 3.
  \end{equation*}
  Consequently, $a \le \sqrt{3} \seq{1/\sqrt{n}}$ and $\seq{1/\sqrt{n}} \le \sqrt{2} a$, and therefore these sequences generate the same characteristic set.

  Finally, we conclude
  \begin{equation*}
    \ideal{C}^2 = \ideal{\diag c^{\star}} = \ideal{\diag a} = \ideal{\diag \seq{1/\sqrt{n}}},
  \end{equation*}
  and therefore
  \begin{equation*}
    \ideal{Z} = \ideal{C} = \ideal1{\diag \sqrt{c^{\star}}} = \ideal{\diag \seq{1/\sqrt[4]{n}}}. \qedhere
  \end{equation*}
\end{proof}

\Cref{lem:anderson-ideal} along with \Cref{ex:brown-trace-zero} lead us to pose the following question regarding minimality of the ideals $\ideal{C}, \ideal{Z}$ arising from a commutator representation $P = \comm{C}{Z}$ of a rank-one projection.

\begin{question}
  \label{que:anderson-minimal}
  Let $P$ be a rank-one projection.
  Do there exist $C,Z$ with $P = \comm{C}{Z}$ for which both $\ideal{C},\ideal{Z} \subseteq \ideal{\diag(1/\sqrt[4]{n}}$ and at least one inclusion is proper?
\end{question}

In the previous question, we explicitly ask for \emph{both} principal ideal inclusions.
The reason for this is that it is possible to obtain an asymmetric solution were one of the principal ideal inclusions is proper, whereas the other doesn't hold.
Indeed, from Anderson \cite[Theorem~3]{And-1977-JRAM}, it is clear that for any $1/2 < t < 1$, one may obtain operators $C_t,Z_t$ for which $P = \comm{C_t}{Z_t}$ and $\ideal{C_t} \subseteq \schatten{2/t} \subsetneq \ideal{\diag\seq{1/\sqrt[4]{n}}}$.
However, in this case one cannot have $Z_t \subseteq \ideal{\diag\seq{1/\sqrt[4]{n}}} \subseteq \bigcap_{p > 4} \schatten{p}$ because it would violate the result of Brown \cite{Bro-1994-JRAM} mentioned in \Cref{ex:brown-trace-zero}.

This problem comes from similar ideas studied in \cite[Section 7]{DFWW-2004-AM} and the open question there for ideals $\I$: If $\comm{\I}{ \BH}_1$ (or $\comm{\I^{1/2}}{ \I^{1/2}}_1$) contains $P$ (or some finite rank nonzero trace operator), must $\diag\seq{1/\sqrt{n}} \in  \I$?
Also worthy of note is Pearcy--Topping's \cite[Problem 2]{PT-1971-MMJ} on Schatten-$p$ commutators: Is $\comm{\schatten{2p}}{ \schatten{2p}}_1 = \schatten{p}$ for $p >1$, to which Brown provided the negative solution \cite{Bro-1994-JRAM} when $p \leq 2$.
The case $p > 2$ remains open.

\subsection*{Matrix splitting and $\BH$-ideals}

For this subsection we mention two special case theorems whose much more general versions are in  \cite[Theorems 4 and 8]{LW-2021-OM} auxiliary to this work, and \Cref{thm:bandable} in particular was used in the proof of \Cref{lem:anderson-ideal} above.

Our universal block tridiagonal matrices mentioned throughout this paper can be described by means of a sequence of orthogonal projections summing to the identity $\seq{P_n}_{n \in \mathbb{N}}$ via its block diagonals $T_0, T_{\pm 1}$, that is, $T = T_{-1}+T_0+T_1$ where $T_0 := \sum_{n=1}^\infty P_nTP_n, T_1 := \sum_{n=1}^\infty P_nTP_{n+1}$ and $T_{-1} := \sum_{n=1}^\infty P_{n+1}TP_n$.
The exponential dimensions of our projections are prescribed in \Cref{t1,T1,t11}, but the next theorem for block tridiagonal matrices also extends naturally to block banded matrices of arbitrary bandwidth and block sizes.

\begin{theorem}[\protect{\cite[Theorem 4]{LW-2021-OM}}]
  \label{thm:bandable}
  Let $T \in \KH$ be represented in any of its universal block tridiagonal forms with $T_0, T_{\pm 1}$ representing its central, upper and lower block diagonal matrices, so $T = T_{-1}+T_0+T_1$.
  Then each $T_0, T_{\pm 1} \in \ideal{T}$ and the following equality of ideals holds:
  \begin{equation*}
    \ideal{T} = \ideal{T_{-1}}+\ideal{T_0}+\ideal{T_1}.
  \end{equation*}
\end{theorem}

An elementary consequence is: Let $CZ - ZC = D$ with $C$ compact and $Z$ bounded.
\begin{enumerate}
\item \label{item:bandable-big-o} Then $D \in \ideal{C}$, i.e., $\ideal{D} \subseteq \ideal{C}$, or equivalently, $s(D) = O(D_ks(C))$ for some $k$-ampliation $D_k$.
\item \label{item:bandable-tridiagonal} In any basis where $C$ is block tridiagonal, i.e., $C = C_{-1} + C_0 + C_1$, one has $C_{-1}, C_0, C_1 \in \ideal{C}$.
  And so each have $s$-number sequences satisfying the above $O$-condition in \ref{item:bandable-big-o} for some $k$.
\item If we simultaneously block tridiagonalize $C, D$ or all three operators (which can be achieved by \Cref{t11}), the same can be said for  each of them as in \ref{item:bandable-tridiagonal}.
\end{enumerate}

Regarding the Pearcy--Topping problem equivalences mentioned earlier (\Cref{equivthm}), we have
\begin{corollary}
  \label{blockdiag corollary}
  For $T = CZ-ZC$ with all three operators compact, when they are simultaneously put in universal block tridiagonal form relative to any basis, one has the constraints: $\ideal{T} \subseteq \ideal{C}\ideal{Z}$ along with $\ideal{T} = \ideal{T_{-1}}+\ideal{T_0}+\ideal{T_1}$.
  And the latter equality holds likewise for $C$ and $Z$.

  Another mentioned equivalence (see \Cref{equivthm}) where diagonal $D = CZ-ZC$, when $C,Z$ are simultaneously put in universal block tridiagonal form relative to the basis diagonalizing $D$, then $D$ is a block $5$-diagonal matrix and one has the constraints $\ideal{D} \subseteq \ideal{C}\ideal{Z}$ where $\ideal{D} = \ideal{D_{-2}} + \ideal{D_{-1}}+\ideal{D_0}+\ideal{D_1} + \ideal{D_2}$, $\ideal{C} = \ideal{C_{-1}}+\ideal{C_0}+\ideal{C_1}$ and $\ideal{T} = \ideal{Z_{-1}}+\ideal{Z_0}+\ideal{Z_1}$.
\end{corollary}

And for general interest, though not used in this paper, \cite[Theorem~9]{LW-2021-OM} handles the more general case where $T$ is a full block matrix, that is, not finite block banded, but the conclusion is not as strong as \Cref{thm:bandable}.
This slightly weaker conclusion obtained in \cite{LW-2021-OM} is, in general, optimal and involves the notion of the \emph{arithmetic mean closure} of the relevant ideal.

In the spirit of \Cref{sec3,sec3a} (esp., \eqref{eq6} and \Cref{T.1} regarding matrix finite full blocks), we close this matrix splitting subsection with a $\BH$-ideal consequence of boundedness of $k_n(\norm{A_n}+\norm{B_n})$ for block tridiagonal $T$ with zero central blocks.

\begin{proposition}
  \label{prop:inclusion-diag-omega}
  Suppose $T \in \BH$ be in block tridiagonal form with zero central blocks, i.e.,
  \begin{equation*}
    \label{eq:zero-block-diag}
    T =
    \begin{pmatrix}
      0   & A_1 & 0   & \cdots          \\
      B_1 & 0   & A_2 & \ddots           \\
      0   & B_2 & 0   &\ddots            \\
      \vdots & \ddots & \ddots  &\ddots \\
    \end{pmatrix},
    \quad \text{ with central blocks of size } k_n \times k_n,
  \end{equation*}
  and let $r_n := \sum_{j=1}^n k_j$ denote the partial sums, and let $r_0 := 0$.
  If the sequences of norms $\seq{\norm{A_n}}, \seq{\norm{B_n}}$ are nonincreasing and the sequence $\seq{r_n (\norm{A_n} + \norm{B_n})}$ is bounded, then $\ideal{T} \subseteq \ideal{\diag\seq{1/n}}$.

  Moreover, if $k_n = \Theta(\rho^n)$ for some $\rho > 1$, then the boundedness condition can be replaced by the boundedness of $\seq{k_n(\norm{A_n} + \norm{B_n})}$.
\end{proposition}

\begin{proof}
  We begin by noting that since $r_n \to \infty$, the boundedness of $\seq{r_n (\norm{A_n} + \norm{B_n})}$ guarantees that $\norm{A_n}, \norm{B_n} \to 0$, and hence $T$ is compact.
  
  As in the proof of \Cref{lem:anderson-ideal} we apply \Cref{thm:bandable} to conclude that $\ideal{T} = \ideal{T_1} + \ideal{T_{-1}}$, so it suffices to prove $\ideal{T_{\pm 1}} \subseteq \ideal{\diag\seq{1/n}}$.
  By symmetry considerations, it suffices to establish this only for $T_1$.
  We note, as in the proof of \Cref{lem:anderson-ideal}, that\footnote{%
    The sequence $s(A_n)$ is \emph{finite} with length $\min\{k_n, k_{n+1}\}$; but we may pad it with zeros to make it have length $k_n$ in case $k_{n+1} < k_n$.}
  $s(T_1) = \left( \bigoplus s(A_n) \right)^{\star}$, and also $s(A_n) \le \seq{\norm{A_n}}_{j=1}^{k_n}$ (the latter is a constant sequence).
  Therefore,
  \begin{equation*}
    s(T_1) = \left( \bigoplus s(A_n) \right)^{\star} \le a := \bigoplus_n \left( \seq{\norm{A_n}}_{j=1}^{k_n} \right).
  \end{equation*}
  The sequence $a = a^{\star}$ is already monotone since the sequences $\seq{\norm{A_n}}, \seq{\norm{B_n}}$ are nonincreasing.

  Then let $M > 0$ denote a bound for $\seq{r_n (\norm{A_n} + \norm{B_n})}$.
  Then, for any $m \in \mathbb{N}$, there is some $n \in \mathbb{N}$ with $r_{n-1} < m \le r_n$, and by the definition of $a$, $a_m = \norm{A_n}$.
  Thus
  \begin{equation*}
    a_m = \norm{A_n} \le \frac{M}{r_n} \le \frac{M}{m}.
  \end{equation*}
  Hence $a$ (and also $s(T_1)$) lies in any characteristic set containing $\seq{1/n}$, so $\ideal{T_1} \subseteq \ideal{\diag \seq{1/n}}$.

  If, in addition, $k_n = \Theta(\rho^n)$ for some $\rho > 1$, we note that $k_n = \Theta(r_n)$ and hence $\seq{r_n(\norm{A_n} + \norm{B_n})}$ is bounded if and only if $\seq{k_n(\norm{A_n} + \norm{B_n}}$ is bounded.
\end{proof}

\bibliographystyle{plain}
\bibliography{references.bib}

\end{document}